\documentclass[11pt,a4paper,reqno]{amsart}
\usepackage{titlesec}
\usepackage{mathrsfs}
\usepackage{amsmath,amssymb}
\usepackage{amsfonts}
\usepackage{latexsym,bm}
\usepackage{amsthm}
\usepackage{cite}
\usepackage{amssymb}
\usepackage{amssymb,amscd}
\usepackage{amsbsy}
\usepackage{fancyhdr,graphicx}
\usepackage{indentfirst}
\usepackage{graphics,color}
\usepackage{epsfig}
\usepackage{xcolor}
\usepackage{overpic}
\usepackage{caption}
\usepackage{subcaption}
\usepackage{multirow}
\usepackage{diagbox}
\usepackage{array}
\usepackage{makecell}
\usepackage{arydshln}
\usepackage{booktabs}
\usepackage{threeparttable}

\usepackage{enumerate}
\titleformat{\section}{\normalfont\Large\bfseries}{\thesection}{1em}{}
\titleformat{\subsection}{\normalfont\large\bfseries}{\thesubsection}{0.5em}{}
\titleformat{\subsubsection}{\normalfont\normalsize\bfseries}{\thesubsubsection}{0.5em}{}

\newtheorem {theorem} {Theorem}[section]
\newtheorem {proposition} {Proposition}[section]

\newtheorem {lemma}  {Lemma}[section]

\newtheorem {definition} {Definition}[section]
\newtheorem {remark} {Remark}[section]
\newtheorem {problem} {Problem}
\theoremstyle{remark}
\theoremstyle{definition}

\newpage
\makeatletter
\def\@settitle{\begin{center}\normalfont\LARGE\bfseries\@title\end{center}\par\vskip 0.5em}
\makeatother
\begin{document}

\title[Limit cycles of generalized Abel equations]
{On the Number of Limit Cycles in Generalized Abel Equations with Coefficients Having the Chebyshev Property}

\author[J. Huang]{Jianfeng Huang}
\address{Department of Mathematics, Jinan University, Guangzhou 510632, P. R. China}
\email{thuangjf@jnu.edu.cn}

\author[R. Tian]{Renhao Tian}
\address{School of Mathematics (Zhuhai), Sun Yat-sen University, Zhuhai 519082, P. R. China}
\email{tianrh5@mail.sysu.edu.cn}

\author[Y. Zhao]{Yulin Zhao}
\address{School of Mathematics (Zhuhai), Sun Yat-sen University, Zhuhai 519082, P. R. China}
\email{mcszyl@mail.sysu.edu.cn}

\begin{abstract}
This paper concerns the maximum number of limit cycles of generalized Abel differential equations $dx/dt = A(t)x^p + B(t)x^q$, where $A$ and $B$ belong to the linear span of a family of functions having the Chebyshev property.
Motivated by a recent open problem posed by Huang et al. (Nonlinearity, 2026), we investigate whether this maximum number can be bounded in terms of $p$, $q$, and the structure of the family.
Under some natural hypotheses and by means of first- and second-order analyses using Melnikov functions, we provide lower bounds for this maximum number.
In contrast to previous work, no specific form for the coefficients is assumed.
We then apply these estimates to Abel equations with trigonometric polynomial, polynomial, and hyperbolic coefficients. 
In the trigonometric polynomial case, we reestablish the results of \'{A}lvarez et al. (J. Math. Anal. Appl., 2008) and Huang et al. (SIAM J. Appl. Dyn. Syst., 2020), while
in the polynomial case, we improve the classical lower bound given by Lins-Neto.
\end{abstract}
\date{}
\maketitle
\section{Introduction}
In the qualitative theory of differential equations, determining the number and distribution of isolated periodic orbits, known as \textit{limit cycles}, is widely considered a cornerstone problem.
This subject goes back to the pioneering work of Poincar\'{e} in the late 19th century, but it became a central focus of modern research after Hilbert proposed his 23 open problems at the International Congress of Mathematicians in 1900.
In particular, the second part of Hilbert's 16th problem asks for the maximum number and the relative positions of limit cycles for planar polynomial differential equations of degree $n$.
Despite decades of intensive research, it remains unsolved even for the simplest non-trivial case $n=2$.
For a comprehensive overview, we refer the reader to the surveys \cite{Ily,Rou,Li}.
Due to the complexity of the original problem, mathematicians have subsequently introduced various weakened versions.
For example, Arnold \cite{Arnold1,Arnold2} proposed studying the maximum number of limit cycles bifurcating from a periodic annulus of a polynomial Hamiltonian system under small perturbations.
In a similar spirit, Smale \cite{Smale0} suggested restricting the investigation of Hilbert's 16th problem to classical polynomial Li\'{e}nard systems.

Another problem closely related to Hilbert's 16th problem, originally proposed by Pugh (see e.g. \cite{Lins,Smale0,Smale1}),
is to investigate the maximum number of \textit{isolated closed solutions} of the following generalized Abel equations:
\begin{equation}\label{eq0}
    \frac{dx}{dt}=\sum_{i=0}^{k}A_{i}(t)x^{i},\quad A_{i}\in C^{0}[-T,T],\quad t\in[-T,T],
\end{equation}
where the closed solutions satisfy the boundary condition $x(-T)=x(T)$.
By convention, an isolated closed solution is called a limit cycle.
The study of equation \eqref{eq0} is of great importance, as numerous families of planar polynomial systems can be brought into this form via appropriate transformations.
Notable examples include {Li\'{e}nard systems \cite{Cherkas,Harko2014}}, quadratic systems \cite{GL,Lins}, rigid systems \cite{openproblem}, and systems with homogeneous nonlinear terms, among others \cite{ABS,GL,HL}.
Several real-world problems can also be reduced to equation \eqref{eq0}.
For instance, an Abel-type equation is used to approximate unstable limit cycles in power converters \cite{appliction1} and to model the relativistic evolution of cosmological fluids \cite{appliction2}.
Furthermore, Abel equations also arise in the analysis of travelling waves in models of glioblastoma growth \cite{appliction3} and in formulations of inflation conditions in Einstein--Friedmann cosmology \cite{appliction4}.

Returning to the study of limit cycles for equation \eqref{eq0},
the maximum number of limit cycles strongly depends on the degree $k$.
When $k = 1$, \eqref{eq0} is a linear equation with an affine return map, which implies that it has at most one limit cycle.
For $k = 2$, \eqref{eq0} reduces to a Riccati equation.
Since its return map is a Möbius transformation, it has at most two limit cycles (see \cite{Lins, Lloyd, Lloyd2}).
However, the situation changes drastically for the classical Abel equation ($k = 3$).
Lins Neto \cite{Lins} proved that without additional restrictions on the coefficients,
the number of limit cycles of the equation is no longer uniformly bounded.
More precisely, by employing bifurcation techniques, he showed that Abel equations of the form
\begin{equation}\label{eq1}
    \frac{dx}{dt}=A(t)x^{3}+B(t)x^{2},\quad A,B\in C^{0}[-T,T],\quad t\in[-T,T],
\end{equation}
can exhibit at least $m$ limit cycles when the coefficients $A(t)$ and $B(t)$ are polynomials or trigonometric polynomials of degree $m$.
Related lower-bound constructions were later extended to generalized Abel equations with $k > 3$ (see, e.g., \cite{GG, AGY, HTV}).

Let $H_{\mathbb{T}}(n,m)$ denote the maximum number of limit cycles of equation \eqref{eq1} when $A$ and $B$ are trigonometric polynomials of degrees $n$ and $m$, respectively.
The problem of determining the exact value of $H_{\mathbb{T}}(n,m)$ is widely known as the Smale-Pugh problem.
A common approach to provide lower bounds for $H_{\mathbb{T}}(n,m)$ relies on small perturbations and Melnikov functions.
Specifically, using a first-order Melnikov analysis, \'{A}lvarez et al. \cite{AGY} proved that $H_{\mathbb{T}}(n,1)\geq n+2$ and $H_{\mathbb{T}}(1,m)\geq 2m+1$.
Later, Huang et al. \cite{HTV} developed a second-order analysis to show that $H_{\mathbb{T}}(n,m)\geq 2(n+m)-1$.
For the low-degree case $n = m = 1$, where both coefficients are linear trigonometric polynomials, equation \eqref{eq1} can be written as
\begin{equation}\label{eq2}
    \frac{dx}{dt}=(a_{1}+a_{2}\sin t+a_{3}\cos t)x^{3}+(b_{1}+b_{2}\sin t+b_{3}\cos t)x^{2},
\end{equation}
where $a_{i},b_{i}\in\mathbb{R}$ for $i=1,2,3$.
For this equation, Yu et al. \cite{YHL} proved that $H_{\mathbb{T}}(1,1)=3$.
Recently, this result was extended in \cite{HTZ}.
The authors proved that if the coefficient functions $A$ and $B$ of equation \eqref{eq1} belong to \(\operatorname{span}(\mathcal F)\), where $\mathcal{F}=\{f_0,f_1,f_2\}$ is an Extended Chebyshev system (ET-system) on $[0,T)$,
with each $f_i$ being continuous on $[0,T]$ and $f_0|_{[0,T]} > 0$, then equation \eqref{eq1} admits at most three limit cycles.
It is worth mentioning that this criterion not only recovers the aforementioned result of \cite{YHL}, but also solves the problem when $A$ and $B$ are quadratic polynomials, among other cases.
For further details, we refer the reader to \cite{HTZ, bravo1, bravo2, Alvarez2007}.

An insightful perspective introduced in \cite{HTZ} is to unify the treatment of standard polynomials and trigonometric polynomials by embedding them into the Chebyshev framework (see Section 2 for specific definitions).
Following this line of research, the authors formulated the following Chebyshev version of the Smale-Pugh problem:
\begin{problem}\label{problem}(\cite{HTZ})
Let $\mathcal{F}=\{f_0,f_1,\cdots,f_n\}$ be a Chebyshev system on $[0,T)$, with each $f_i$ continuous on $[0,T]$ and $f_0|_{[0,T]} > 0$.
Is the maximum number of limit cycles for equation \eqref{eq1} with coefficients $A,B\in\operatorname{span}(\mathcal{F})$ bounded in terms of $n$?
Here $\operatorname{span}(\mathcal{F}):=\{\sum_{i=0}^{n}\lambda_{i} f_i \mid \lambda_i \in \mathbb{R}\}$.
\end{problem}
For the case $n\leq 2$, the problem has already been studied in \cite{Alvarez2007,HTZ}.
It is therefore natural to ask whether the Chebyshev framework can yield estimates when the coefficient space has dimension greater than three, equivalently when \(n>2\).
{Our main objective is to establish lower bounds for the maximum number of limit cycles in Problem \ref{problem} for arbitrary $n$, by developing a unified framework based on first- and second-order Melnikov analyses combined with the Chebyshev property.}

Let us briefly outline the structure of this paper.
We start by introducing the setting and the main results in Section 2.
Next, we present several preliminary results in Section 3, including the expressions of the first- and second-order Melnikov functions and some properties of Chebyshev systems.
In Section 4, we provide the complete proofs of our results.
{We then study Abel equations with several important families of coefficient functions, including trigonometric polynomials, polynomials, and hyperbolic functions, to illustrate the applications of our criterion and discuss their scope in Section 5.}
Finally, we conclude the paper with a discussion on future research directions in Section 6.

\section{Setting and Main Results}

\subsection{Problem Setup}\label{setting}
Let \(\mathcal F=\{f_i\}_{i=0}^{\infty}\) be a sequence of real-analytic functions defined on an open neighborhood of \([-T,T]\), where $f_{0}(t)> 0$ for all $t \in [-T,T]$.
We denote the $(n+1)$-th truncation of $\mathcal{F}$ by $\mathcal{F}_{n}=\{f_{0},f_{1},\cdots,f_{n}\}$.
In this paper, we study the generalized Abel differential equation of the form
\begin{equation}\label{Abel equation}
    \frac{dx}{dt}=A(t)x^{p}+B(t)x^{q},
\end{equation}
where $p, q \in \mathbb{Z}_{\geq 2}$ with $p \neq q$, and the coefficients $A \in \operatorname{span}(\mathcal{F}_{n})$ and $B \in \operatorname{span}(\mathcal{F}_{m})$.
In particular, when $p=3$ and $q=2$, equation \eqref{Abel equation} reduces to the classical Abel equation \eqref{eq1}.
For fixed $p$ and $q$, let $\mathcal{H}_{p,q}(n,m)$ denote the Hilbert number of equation \eqref{Abel equation}, defined as the supremum of the number of limit cycles that equation \eqref{Abel equation} can have for any $A\in\operatorname{span}(\mathcal{F}_{n})$ and $B\in\operatorname{span}(\mathcal{F}_{m})$, where this supremum may be \(+\infty\).

When studying $\mathcal{H}_{p,q}(n,m)$, without loss of generality, we may assume that
\begin{equation}\label{eq:integral_zero}
    \int_{-T}^{T}f_{k}(t)\,dt=0\quad \text{for}\quad k\geq 1.
\end{equation}
Indeed, since $f_{0}$ does not change sign on $[-T,T]$, there exists a unique constant $c_{k}$ for each $k\geq 1$ such that
\begin{equation*}
    \int_{-T}^{T} \big( f_{k}(t)-c_{k}f_{0}(t) \big) \,dt=0.
\end{equation*}
Since $\operatorname{span}(\mathcal{F}_{n})=\operatorname{span}\{ f_{0}, f_{1}-c_{1}f_{0}, \cdots, f_{n}-c_{n}f_{0} \}$, one can always replace each $f_{i}$ with the new basis vectors $f_{i}-c_{i}f_{0}$ for $i \geq 1$ to ensure \eqref{eq:integral_zero}.

For each \(i\geq0\), decompose \(f_i\) into its even and odd parts.
That is,
\begin{equation*} 
    f_{i}^{e}(t)=\frac{f_{i}(t)+f_{i}(-t)}{2} \quad \text{and} \quad f_{i}^{o}(t)=\frac{f_{i}(t)-f_{i}(-t)}{2}, 
\end{equation*} respectively. 
We also introduce their corresponding primitives
\begin{equation*}
F_{i}^{e}(t)=\int_{-T}^{t}f_{i}^{e}(s)\,ds \quad \text{and} \quad F_{i}^{o}(t)=\int_{-T}^{t}f_{i}^{o}(s)\,ds.
\end{equation*}
It is easy to check that for every $i\geq 1$, $F_{i}^{o}$ is an even function, whereas $F_{i}^{e}$ is an odd function.

To refine our analysis, we introduce the index sets
\begin{equation*}
    \begin{split}
        &\mathcal{N}_{k}=\{0,1,\cdots,k\},\quad \mathcal{N}^{\text{odd}}_{k}=\{i\in\mathcal{N}_{k} \mid f_{i}^{e}\equiv 0 \text{ on } [-T,T]\},\\
        &\mathcal{N}^{\text{non-odd}}_{k}=\{i\in\mathcal{N}_{k} \mid f_{i}^{e}\not\equiv 0 \text{ on } [-T,T]\},
    \end{split}
\end{equation*}
and let $k_{1}=\#\mathcal{N}^{\text{non-odd}}_{k}-1$, where $\# A$ denotes the cardinality of set $A$.
From the above definition, one immediately obtains the following facts:
\begin{itemize}
    \item For each $k\in\mathbb{N}$, the integer $k_{1} \in \mathbb{N}_{\leq k}$ is uniquely determined and nondecreasing with respect to $k$.
    \item For any $i\in\mathcal{N}^{\text{non-odd}}_{k}$, $f_{i}^{e}\not\equiv 0$, meaning that $f_{i}$ is not an odd function.
    \item $0\in\mathcal{N}^{\text{non-odd}}_{k}$, since $f_{0}$ does not change sign on $[-T,T]$.
\end{itemize}

For convenience and to maintain consecutive indexing, the nonzero even and odd parts are separately relabeled according to the increasing order of their original indices.
That is,
\begin{equation*}
    \begin{split}
    &\{f^{e}_{i}\}_{i\in\mathcal{N}^{\text{non-odd}}_{k}}:=\{\tilde{f}^{e}_{0},\tilde{f}^{e}_{1},\cdots,\tilde{f}^{e}_{k_{1}}\}, \quad
    \{f^{o}_{i}\}_{i\in\mathcal{N}^{\text{odd}}_{k}}:=\{\tilde{f}^{o}_{1},\cdots,\tilde{f}^{o}_{k-k_{1}}\},\\
    &\{F^{e}_{i}\}_{i\in\mathcal{N}^{\text{non-odd}}_{k}}:=\{\tilde{F}^{e}_{0},\tilde{F}^{e}_{1},\cdots,\tilde{F}^{e}_{k_{1}}\}.
    \end{split}
\end{equation*}

To obtain lower bounds for \(\mathcal H_{p,q}(n,m)\), consider the second-order perturbation of the generalized Abel differential equation \eqref{Abel equation} given by
\begin{equation}\label{perturbation equation}
    \frac{dx}{dt}=\left(f_{1}(t)+\varepsilon P_{1}(t)+\varepsilon^{2} P_{2}(t)\right)x^{p}+\left(\varepsilon Q_{1}(t)+\varepsilon^{2} Q_{2}(t)\right)x^{q},
\end{equation}
where $p,q\in\mathbb{Z}_{\geq 2}$ with $p\not=q$, and the perturbation terms belong to the truncated spaces $\operatorname{span}(\mathcal{F}_{n})$ and $\operatorname{span}(\mathcal{F}_{m})$, respectively, namely
\begin{equation*}
    P_{i}(t)=\sum_{j=0}^{n}a_{ij}f_{j}(t)\in \operatorname{span}(\mathcal{F}_{n}) \quad \text{and} \quad Q_{i}(t)=\sum_{j=0}^{m}b_{ij}f_{j}(t)\in \operatorname{span}(\mathcal{F}_{m})
\end{equation*}
for $i=1,2$, and $n,m\geq 1$.

The parameter space associated with equation \eqref{perturbation equation} is $\mathbb{R}^{2(n+m+2)}$. For convenience, we collect all perturbation coefficients into a single vector $\mu=(\mathbf{a}_{1},\mathbf{a}_{2},\mathbf{b}_{1},\mathbf{b}_{2})$, where
\begin{equation*}
    \mathbf{a}_{i}=(a_{ij} \mid j=0,1,\cdots,n) \quad \text{and} \quad \mathbf{b}_{i}=(b_{ij} \mid j=0,1,\cdots,m) \quad \text{for } i=1,2.
\end{equation*}
Let $x_{\varepsilon}(t,\rho;\mu)$ denote the solution of the equation satisfying the initial condition $x_{\varepsilon}(-T,\rho;\mu)=\rho$.
For $\varepsilon=0$, we have the unperturbed system
\begin{equation*}
    \frac{dx}{dt}=f_{1}(t)x^{p}.
\end{equation*}
A straightforward integration yields its general solution
\begin{equation}\label{x0(t)}
    x_{0}(t,\rho;\mu)=\frac{\rho}{\big(1+(1-p)\rho^{p-1}F_{1}(t)\big)^{\frac{1}{p-1}}},
\end{equation}
where $F_{1}(t)=\int_{-T}^{t}f_{1}(s)\,ds$. Then, there exists a maximal interval of initial values, $J$, for which $x_0(t,\rho;\mu)$ is well-defined on $[-T,T]$. More precisely,
\begin{equation}\label{J}
    J = \{ \rho\in\mathbb{R} \mid 1+(1-p)\rho^{p-1}F_{1}(t) > 0 \text{ for all } t\in[-T,T] \}.
\end{equation}
It is clear that for any $\rho\in J$, the unperturbed solution satisfies $x_{0}(T,\rho)=\rho$. 
By virtue of the analytic dependence of solutions on both initial conditions and parameters, the displacement function associated with the return map can be expanded in powers of $\varepsilon$ as
\begin{equation*}
    x_{\varepsilon}(T,\rho;\mu)-\rho = \sum_{i=1}^{\infty} \varepsilon^i M_{i}(\rho;\mu),
\end{equation*}
where each coefficient $M_{i}(\rho;\mu)$ is an analytic function on $J\times\mathbb{R}^{2(n+m+2)}$. As usual, $M_{i}$ is called the \textit{$i$-th order Melnikov function} associated with the perturbed equation \eqref{perturbation equation}.

Let \(M_k\) be the first non-identically vanishing Melnikov function. Since
the displacement function is analytic in \((\rho,\varepsilon)\), the normalized
displacement function
\[
\widehat D(\rho,\varepsilon)
:=
\frac{x_{\varepsilon}(T,\rho;\mu)-\rho}{\varepsilon^k},
\qquad \varepsilon\neq0,
\]
extends analytically to \(\varepsilon=0\) by setting
\[
\widehat D(\rho,0):=M_k(\rho;\mu).
\]
If \(\rho_0\in J\) is an isolated zero of \(M_k(\cdot;\mu)\) of finite
multiplicity \(s\), then the Weierstrass Preparation Theorem yields a
neighborhood \(U\subset J\) of \(\rho_0\) such that, for every sufficiently
small \(\varepsilon\neq0\), the number of zeros of the displacement function
in \(U\), counted with multiplicities, is at most \(s\). 
In particular, if \(\rho_0\) is a simple zero of \(M_k(\cdot;\mu)\), then the Implicit Function Theorem gives a unique branch \(\rho=\rho(\varepsilon)\), with \(\rho(0)=\rho_0\), of fixed points of the Poincar\'e map for every sufficiently small \(\varepsilon\neq0\). Moreover,
\[
\partial_{\rho}\!\left(x_{\varepsilon}(T,\rho;\mu)-\rho\right)
\big|_{\rho=\rho(\varepsilon)}\neq0,
\]
so the corresponding limit cycle is hyperbolic.

Consequently, if \(M_k(\cdot;\mu)\) has finitely many distinct simple zeros contained in a compact subinterval $I\subset J$, the preceding argument can be applied simultaneously in pairwise disjoint neighborhoods of these zeros.
This local persistence result is the mechanism used below to establish the
lower bounds for \(\mathcal H_{p,q}(n,m)\).
\subsection{Main results}

As noted in the Introduction, the coefficients of the perturbed equation \eqref{perturbation equation} are not restricted to polynomial or trigonometric forms, but are instead general functions satisfying specific Chebyshev properties.
To establish our framework, we begin by recalling the definition of a \textit{complete Chebyshev system} (CT-system).
For more details on Chebyshev systems, see Section \ref{3.2}.
\begin{definition}
    Let $\{f_{0},f_{1},\cdots,f_{n}\}$ be a set of analytic functions on an open interval $I\subset \mathbb{R}$.
    This set is a complete Chebyshev system (CT-system) on $I$ if for every $k\in\{0, 1, \cdots, n\}$, any nontrivial linear combination
    $$\alpha_{0} f_{0}(x) + \alpha_{1} f_{1}(x) + \cdots+ \alpha_{k}f_{k}(x)$$ has at most $k$ isolated zeros on $I$.
\end{definition}

For \(r\geq1\), define
\[
\mathcal B_r
=
\operatorname{span}
\big(\{
\tilde F_1^e\tilde f_1^o,\ldots,
\tilde F_r^e\tilde f_1^o
\}\big),
\
\mathcal B_0=\{0\}.
\]
For the fixed truncation indices \(n,m\geq 1\), let
\[
m_1=\#\mathcal N_m^{\mathrm{non\text{-}odd}}-1,
\
n_1=\#\mathcal N_n^{\mathrm{non\text{-}odd}}-1,
\]
and assume \(m_1,n_1\geq1\).
To facilitate the analysis, we impose the following hypotheses on the functions involved:
\begin{itemize}
    \item[(H$_{1}$)] $f_{1}$ is an odd function, i.e. $f_{1}=\tilde{f}_{1}^{o}$, and $f_{1}\geq 0$ with $f_{1}\not\equiv 0$ on $[0,T]$.
    \item[(H$_{2}$)] $\{f^{e}_{i}\}_{i\in\mathcal{N}^{\text{non-odd}}_{r}}$  is a CT-system on $(0,T)$ for all $r\in \mathbb{N}_{\geq 1}$.
    \item[{(H$_{3}$)}] {For all $k,l\in\mathbb{Z}_{\geq 1}$, $\tilde{F}^{e}_{l}\tilde{f}^{o}_{k}\in\mathcal{B}_{k+l-1}\setminus\mathcal{B}_{k+l-2}$, and $m_{1}\geq \min\{m-m_{1}-1,n_{1}-1\}$.}
\end{itemize}
Our first main result concerns the first-order Melnikov function $M_{1}(\cdot\,;\mu)$ of equation \eqref{perturbation equation} under hypotheses (H$_{1}$) and (H$_{2}$).
{The parity of $p$ and $q$ determines the symmetry properties of the Melnikov functions and leads to different estimates for the number of their zeros.}
\begin{theorem}\label{thm1}
    Assume that $M_{1}(\cdot\,;\mu)\not\equiv0$ is the first-order Melnikov function of equation \eqref{perturbation equation} and hypotheses (H$_{1}$) and (H$_{2}$) hold.
    Let $Z(M_{1})^{+}$ (resp. $Z(M_{1})^{-}$) be the number of zeros of $M_{1}(\cdot\,;\mu)$ on $J\cap\mathbb{R}^{+}$ (resp. $J\cap\mathbb{R}^{-}$) taking multiplicities into account.
    Then the following statements hold:
    \begin{itemize}
        \item[(i)] If $p$ and $q$ are both odd, then $Z(M_{1})^{\pm}\leq m_{1}+1$.
        Moreover, there exists $\mu_{0}\in\mathbb{R}^{2(n+m+2)}$ such that $M_{1}(\cdot\,;\mu_{0})$ has $2(m_{1}+1)$ simple zeros in $J\setminus\{0\}$;
        \item[(ii)] If $p$ is odd and $q$ is even, then $Z(M_{1})^{\pm}\leq m_{1}+1$ and the two equalities cannot hold simultaneously. Moreover,
\begin{itemize}
            \item[(ii.1)] when $p<q$, there exists $\mu_{0}\in\mathbb{R}^{2(n+m+2)}$ such that $M_{1}(\cdot\,;\mu_{0})$ has $2m_{1}+1$ simple zeros in $J\setminus\{0\}$.
            \item[(ii.2)] when $p>q$, there exists $\mu_{0}\in\mathbb{R}^{2(n+m+2)}$ such that $M_{1}(\cdot\,;\mu_{0})$ has $2m_{1}$ simple zeros in $J\setminus\{0\}$.
    \end{itemize}
        \item[(iii)] If $p$ is even, $q$ is odd, and $p<q$, then $Z(M_{1})^{+}+Z(M_{1})^{-}\leq m_{1}+1$.
            Moreover, there exists $\mu_{0}\in\mathbb{R}^{2(n+m+2)}$ such that $M_{1}(\cdot\,;\mu_{0})$ has $m_{1}+1$ simple zeros in $J\setminus\{0\}$.
        \item[(iv)] If $p$ is even and either $q$ is even or $p>q$, then $Z(M_{1})^{+}+Z(M_{1})^{-}\leq m_{1}+2$.
            Moreover, there exists $\mu_{0}\in\mathbb{R}^{2(n+m+2)}$ such that $M_{1}(\cdot\,;\mu_{0})$ has $m_{1}+2$ simple zeros in $J\setminus\{0\}$.
    \end{itemize}
\end{theorem}
Next, we study the case $M_{1}=0$ and $M_{2}\not =0$ under hypotheses (H$_{1}$), (H$_{2}$) and (H$_{3}$).
\begin{theorem}\label{thm2}
    Assume that $M_{1}(\cdot\,;\mu)$ vanishes identically, $M_{2}(\cdot\,;\mu)\not\equiv0$ is the second-order Melnikov function of equation \eqref{perturbation equation} and hypotheses (H$_{1}$), (H$_{2}$) and (H$_{3}$) hold.
    Let $K=\max\{m-m_{1}+n_{1}, m_{1}+1\}$ and $Z(M_{2})^{+}$ (resp. $Z(M_{2})^{-}$) be the number of zeros of $M_{2}(\cdot\,;\mu)$ on $J\cap\mathbb{R}^{+}$ (resp. $J\cap\mathbb{R}^{-}$) taking multiplicities into account.
    Then the following statements hold:
    \begin{itemize}
        \item[(i)] If $p$ and $q$ are both odd, then $Z(M_{2})^{\pm}\leq K$.
        Moreover, there exists $\mu_{0}\in\mathbb{R}^{2(n+m+2)}$ such that $M_{2}(\cdot\,;\mu_{0})$ has $2K$ simple zeros in $J\setminus\{0\}$;
        \item[(ii)] If $p$ is odd and $q$ is even, then $Z(M_{2})^{\pm}\leq K$ and the two equalities cannot hold simultaneously. Moreover,
\begin{itemize}
            \item[(ii.1)] when $p<q$, there exists $\mu_{0}\in\mathbb{R}^{2(n+m+2)}$ such that $M_{2}(\cdot\,;\mu_{0})$ has $2K-1$ simple zeros in $J\setminus\{0\}$.
            \item[(ii.2)] when $p>q$, there exists $\mu_{0}\in\mathbb{R}^{2(n+m+2)}$ such that $M_{2}(\cdot\,;\mu_{0})$ has $2(K-1)$ simple zeros in $J\setminus\{0\}$.
    \end{itemize}
        \item[(iii)] If $p$ is even, $q$ is odd, and $p<q$, then $Z(M_{2})^{+}+Z(M_{2})^{-}\leq K$.
            Moreover, there exists $\mu_{0}\in\mathbb{R}^{2(n+m+2)}$ such that $M_{2}(\cdot\,;\mu_{0})$ has $K$ simple zeros in $J\setminus\{0\}$.
        \item[(iv)] If $p$ is even and either $q$ is even or $p>q$, then $Z(M_{2})^{+}+Z(M_{2})^{-}\leq K+1$.
            Moreover, there exists $\mu_{0}\in\mathbb{R}^{2(n+m+2)}$ such that $M_{2}(\cdot\,;\mu_{0})$ has $K+1$ simple zeros in $J\setminus\{0\}$.
    \end{itemize}
\end{theorem}
\begin{remark}
We have the following comments on Theorems \ref{thm1} and \ref{thm2}.
    \begin{itemize}
\item[(a)] As will be seen from the proofs of the theorems, (H$_{1}$) and (H$_{2}$) provide the fundamental hypotheses for obtaining the explicit estimates for the maximum number of zeros of the Melnikov functions $M_{1}$ and $M_{2}$ of equation \eqref{perturbation equation}. 
It is also worth noting that, as the higher-order analysis proceeds, more information on the algebraic structure of the vector space spanned by $\mathcal{F}$ (or $\mathcal{F}_{n}$ and $\mathcal{F}_{m}$) is needed. Indeed, from the theorems, it is sufficient to consider $M_{1}$ under hypotheses (H$_{1}$) and (H$_{2}$), whereas the analysis of $M_{2}$ additionally requires hypothesis (H$_{3}$). We illustrate this point in Section \ref{Applications} by considering equation \eqref{perturbation equation} with $\mathcal{F}$ chosen from families of trigonometric polynomials, polynomials and hyperbolic functions.
\item[(b)] The inequality in hypothesis (H$_3$), $m_1\geq \min\{m-m_1-1,n_1-1\}$, is rather mild and is satisfied by a broad class of natural perturbation spaces. 
Indeed, odd and non-odd basis functions commonly occur in comparable numbers within each truncation. 
In particular, if $\mathcal F_m$ contains at least as many non-odd functions as odd functions, then $m_1+1\geq m-m_1$ and hence the above inequality holds automatically, independently of $n_1$.
This balanced parity distribution occurs in all three coefficient families considered in Section~5. 
Thus, the inequality in (H$_3$) does not impose a substantial restriction.
\item[(c)] The two theorems show that the maximum numbers of zeros of $M_{1}$ and $M_{2}$ are closely related to the distribution of odd and non-odd functions in $\mathcal{F}$.
Moreover, according to our approach, the study reduces to a simple counting problem: determining the number of non-odd functions in $\mathcal{F}_{n}$ and $\mathcal{F}_{m}$.
\item[(d)] It is interesting to compare the maximum numbers of zeros of $M_1$ and $M_{2}$. Clearly, the estimate obtained for $M_2$ is greater than (resp. coincides with) that obtained for $M_1$ when $m - m_{1} + n_{1} > m_{1} + 1$ (resp. $m - m_{1} + n_{1} \leq m_{1} + 1$). Notice that $m-m_1$ and $m_1+1$ represent, respectively, the numbers of odd and non-odd functions in $\mathcal F_m$. Thus, in contrast to the analysis of $M_1$, the number of odd functions in $\mathcal F_m$ becomes an important factor and can lead to an increase in the number of zeros of the higher-order Melnikov functions.
\end{itemize}
\end{remark}

The preceding results yield the following lower bounds for \(\mathcal H_{3,2}(n,m)\).

\begin{theorem}\label{thm3}
    Let $n,m\geq 1$, $m_{1}=\#\mathcal{N}^{\text{non-odd}}_{m}-1$ and $n_{1}=\#\mathcal{N}^{\text{non-odd}}_{n}-1$.
    The Hilbert number for the Abel differential equation \eqref{Abel equation} with $p=3$ and $q=2$ satisfies
    \begin{itemize}
        \item[(i)] $\mathcal{H}_{3,2}(n,m)\geq 2m_{1}+1$ if hypotheses (H$_{1}$) and (H$_{2}$) hold;
        \item[(ii)] $\mathcal{H}_{3,2}(n,m)\geq 2\max\{m-m_{1}+n_{1}, m_{1}+1\}-1$ if hypotheses (H$_{1}$), (H$_{2}$) and (H$_{3}$) hold.
    \end{itemize}
\end{theorem}

The results of this section provide a lower bound counterpart to Problem \ref{problem} by quantifying the number of limit cycles that can be realized in general Chebyshev spaces.
In contrast to previous results restricted to low-dimensional spaces or specific families of coefficients, our framework applies to coefficient spaces of arbitrary dimensions under suitable Chebyshev hypotheses. 
The applications to several representative coefficient families will be presented in Section 5.

\section{Preliminaries}
This section collects several preliminary results that will be used in the proof of our main theorems. 
We first derive the expressions of the first- and second-order Melnikov functions of equation \eqref{perturbation equation}, and then recall some useful properties of Chebyshev systems.
\subsection{Melnikov functions for equation \eqref{perturbation equation}}
We begin by introducing a useful result from \cite{HTV}.
This result provides a convenient formula for the Melnikov functions of perturbations around a one-dimensional periodic annulus and will be applied to equation \eqref{perturbation equation}.
Consider the periodic perturbed differential equation of the form
\begin{equation}\label{periodic perturbed equation}
    \frac{dx}{dt}=h(x)f(t)+H(t,x;\varepsilon),
\end{equation}
where
\begin{itemize}
    \item $h$ is analytic on $\mathbb{R}$ with $h(0)=0$, $h\not\equiv 0$,
    \item $f$ is a $2T$-periodic analytic function with $\int_{-T}^{T}f(t)dt=0$,
    \item $H$ is an analytic function on $\mathbb{R} \times \mathbb{R} \times ( -\varepsilon_{0}, \varepsilon_{0})$, for some $\varepsilon_{0} > 0$, such that $t \mapsto H(t, x;\varepsilon)$ is $2T$-periodic and $H(t, x; 0) \equiv 0$.
    Moreover, \(H\) has the Taylor expansion
    \begin{equation*}
        H(t,x;\varepsilon)=\sum_{i=1}^{\infty}l_{i}(t,x)\varepsilon^{i}.
    \end{equation*}
\end{itemize}
Under the above conditions, the unperturbed equation \eqref{periodic perturbed equation}$\mid_{\varepsilon=0}$ possesses a periodic annulus around $x=0$.
Consequently, there exists an open interval $V$ containing $0$, such that $x_{0}(T,\rho)=\rho$ for any $\rho\in V$, where $x_{0}(t,\rho)$ is the solution of \eqref{periodic perturbed equation}$\mid_{\varepsilon=0}$ with initial condition $x_{0}(-T,\rho)=\rho$.
By the analytic dependence of solutions with respect to initial conditions and parameters, we can expand $x(T,\rho;\varepsilon)$ in a power series in $\varepsilon$ as
\begin{equation*}
    x(T,\rho;\varepsilon)=\rho+\sum_{i=1}^{\infty}M_{i}(\rho)\varepsilon^{i}.
\end{equation*}
The following result from \cite{HTV} provides the expressions of the first- and second-order Melnikov functions for equation \eqref{periodic perturbed equation}.
\begin{lemma}(\cite{HTV})\label{lemma1}
     The first-order Melnikov function of equation \eqref{periodic perturbed equation} is
     \begin{equation*}
        M_{1}(\rho;\mu)=h(\rho)S_{1}(T,\rho),
     \end{equation*}
     where $\rho\in V\setminus\{x\mid h(x)=0\}$ and 
     \begin{equation*}
        S_{1}(t,\rho)=\int_{-T}^{t}\frac{l_{1}(s,x)}{h(x)}\mid_{x=x_{0}(s,\rho)}ds.
     \end{equation*}
If \(M_1\equiv0\), the second-order Melnikov function is
          \begin{equation*}
        M_{2}(\rho;\mu)=h(\rho)S_{2}(T,\rho),
     \end{equation*}
     where $\rho\in V\setminus\{x\mid h(x)=0\}$ and
    \begin{equation*}
        S_{2}(t,\rho)=\int_{-T}^{t}\left(\frac{l_{2}(s,x)}{h(x)}+h(x)\cdot\partial_{x}\left(\frac{l_{1}(s,x)}{h(x)}\right)\cdot S_{1}(s,\rho)\right)\Big|_{x=x_{0}(s,\rho)}ds.
    \end{equation*}
\end{lemma}
Inspection of the proof in \cite{HTV} shows that periodicity in \(t\) is used only to interpret the endpoint map as a return map, not in the derivation of the formulas. 
The same calculation therefore applies on the finite interval \([-T,T]\).
Hence, we may apply Lemma \ref{lemma1} to equation \eqref{perturbation equation} and obtain the following result.
\begin{lemma}\label{Melnikovfunction}
    The first-order Melnikov function of equation \eqref{perturbation equation} is
    \begin{equation*}
        M_{1}(\rho;\mu)=\rho^{p}\int_{-T}^{T}\left(P_{1}(t)+Q_{1}(t)x^{q-p}_{0}(t,\rho)\right)dt,
    \end{equation*}
    where $x_{0}(t,\rho)$ is given in \eqref{x0(t)}.
    When $M_{1}\equiv 0$, the second-order Melnikov function of equation \eqref{perturbation equation} can be written as
    \begin{equation*}
        M_{2}(\rho;\mu)=\rho^{p}\int_{-T}^{T}\left(P_{2}(t)+Q_{2}(t)x^{q-p}_{0}(t,\rho)\right)dt+(q-p)\rho^{p}\int_{-T}^{T}\left(Q_{1}(t)S(t,\rho)x_{0}^{q-1}(t,\rho)\right)dt,
    \end{equation*}
    where 
    \begin{equation*}
        S(t,\rho)=\int_{-T}^{t}\left(P_{1}(s)+Q_{1}(s)x^{q-p}_{0}(s,\rho)\right)ds.
    \end{equation*}
\end{lemma}
\begin{proof}
The result follows directly from Lemma \ref{lemma1} applied to equation \eqref{perturbation equation}, with the identifications
    $h(x)=x^{p}$, $f(t)=f_{1}(t)$, $l_{i}(t,x)=P_{i}(t)x^{p}+Q_{i}(t)x^{q}$, $i=1,2$.
\end{proof}

\subsection{Chebyshev systems}\label{3.2}
In this subsection, we present a brief overview of Chebyshev systems, together with some of their useful properties.
For further information, we refer the reader to \cite{KS,Kru,Maz}.
\begin{definition}\cite{KS}
    Let $(f_{0},f_{1},\cdots,f_{n})$ be a set of analytic functions on an open interval $I\subset \mathbb{R}$.
    This set is an extended complete Chebyshev system (ECT-system) on $I$ if for every $k\in\{0, 1, \cdots, n\}$, any nontrivial linear combination
        $$\alpha_{0} f_{0}(x) + \alpha_{1} f_{1}(x) + \cdots+ \alpha_{k}f_{k}(x)$$ 
        has at most $k$ isolated zeros on $I$ counted with multiplicities.    
\end{definition}
\begin{remark}
If $(f_0,\ldots,f_n)$ is an ECT-system on $I$, then, for any $n$ distinct prescribed points $x_1,\ldots,x_n\in I$, there exists a nontrivial linear combination of $f_0,\ldots,f_n$ having exactly these $n$ zeros, all of which are simple; see \cite{KS}. 
This realization property is frequently used in applications of ECT-systems to limit cycle problems, since it allows one to prescribe simple zeros of the corresponding Melnikov functions.
\end{remark}
Clearly, if an ordered set of functions forms an ECT-system, then it is naturally also a CT-system. 
The following definition introduces the Wronskian determinant, which provides a convenient criterion for determining whether a given set of functions constitutes an ECT-system (see, for instance, \cite{KS,Maz}).
\begin{definition}
    Let $f_{0},f_{1},\cdots,f_{n}$ be analytic functions on an open interval $I\subset \mathbb{R}$.
    Then 
    \begin{equation*}
        W[f_0, f_1, \dots, f_n](x) = \det \left( f_j^{(i)}(x) \right)_{0 \le i,j \le n} = \begin{vmatrix}
 f_0(x) & \cdots & f_n(x) \\
 f_0'(x) & \cdots & f_n'(x) \\
 \vdots & \ddots & \vdots \\
 f_0^{(n)}(x) & \cdots & f_n^{(n)}(x)
\end{vmatrix}
    \end{equation*}
    is the Wronskian of $(f_{0},f_{1},\cdots,f_{n})$ at $x\in I$.
\end{definition}

\begin{lemma}(\cite{KS,Maz})
    $(f_{0},f_{1},\cdots,f_{n})$ is an ECT-system on an open interval $I \subset  \mathbb{R}$  if and only if
for each $k = 0, 1,\cdots, n$,
\begin{equation*}
    W[f_0, f_1, \dots, f_k](x)\not=0\text{ for all }x\in I.
\end{equation*}
\end{lemma}

\begin{lemma}(\cite{Kru})\label{chebyshev1}
     The set of ordered functions $(1,f_{1},\cdots,f_{n})$ is an ECT-system on an interval $I$ if and only if $(f^{'}_{1},\cdots,f^{'}_{n})$ is an ECT-system on $I$.
\end{lemma}

\begin{lemma}(\cite{Kru})\label{chebyshev2}
      Let $\eta$ be a smooth non-vanishing function on an interval $I$. 
      Then the set of ordered functions $(\eta f_{0},\eta f_{1},\cdots,\eta f_{n})$ is an ECT-system on $I$ if and only if $(f_{0},f_{1},\cdots,f_{n})$ is an ECT-system on $I$.
\end{lemma}

The following result establishes a criterion for deriving the Chebyshev property of parameter-dependent definite integrals by analyzing the properties of Chebyshev systems of the integrand.
\begin{lemma}(\cite{GGM})\label{chebyshev3}
     Let $n \in \mathbb{N}$, $a,b \in \mathbb{R}$, $\alpha \in \mathbb{R} \setminus \mathbb{N}$, and let $g$ be a monotone and continuous function on $[a,b]$. 
     Denote by $L$ the open interval given by the connected component of the set $\{h \in \mathbb{R}\mid 1 + hg(x) > 0 \text{ for all } x \in [a,b]\}$ which contains the origin. 
     For $h \in L$, we consider the analytic functions
$$J_i(h) = \int_a^b f_i(x)(1 + hg(x))^\alpha dx,$$
where $f_0, f_1, \cdots, f_n$ are analytic on $(a,b)$ such that $(f_0, f_1, \cdots, f_n)$ is a CT-system in $(a,b)$. 
Then $(J_0, J_1, \cdots, J_n)$ is an ECT-system in $L$.
\end{lemma}

For later use, define
\begin{equation}\label{Definition of T}
    \mathcal{T}_{i,\alpha}(y)=\int_{0}^{T} \frac{\tilde{f}^{e}_{i}(t)}{\left(1+(1-p)yF_{1}(t)\right)^{\alpha}}\, dt\text{ for all } y\in J^{'},
\end{equation}
where $\alpha\in\mathbb{R}$ and $J^{'}$ is the connected component of $\{y\in\mathbb{R}\mid 1+(1-p)yF_{1}(t)>0\text{ for all }t\in[-T,T]\}$ containing the origin.
Then we have the following lemma.
\begin{lemma}\label{964}
    Let $\alpha\in\mathbb R\setminus\mathbb Z_{\leq0}$,
$k\in\mathbb Z_{\geq0}$, and hypotheses $(H_1)$ and $(H_2)$ hold.
Then the following statements are valid.
    \begin{itemize}
        \item[(i)]  $\left(1,y^{\alpha}\mathcal{T}_{0,\alpha}(y),\cdots,y^{\alpha}\mathcal{T}_{k,\alpha}(y)\right)$ is an ECT-system on $J^{'}\cap(0,+\infty)$;
        \item[(ii)] $\left(1,(-y)^{\alpha}\mathcal{T}_{0,\alpha}(y),\cdots,(-y)^{\alpha}\mathcal{T}_{k,\alpha}(y)\right)$ is an ECT-system on $J^{'}\cap(-\infty,0)$.
    \end{itemize}
\end{lemma}
\begin{proof}
We only prove statement (i), and the other one follows by the same argument.
For any $y\in J^{'}\cap(0,+\infty)$ and $i\in\mathbb{Z}_{\geq 0}$, a direct computation yields
    \begin{equation*}
        \begin{split}
           (y^{\alpha}\mathcal{T}_{i,\alpha}(y))'=&\alpha y^{\alpha-1}\mathcal{T}_{i,\alpha}(y)+\alpha y^{\alpha}\int_{0}^{T}\frac{(p-1)\tilde{f}^{e}_{i}(t)F_{1}(t)}{\left(1+(1-p)yF_{1}(t)\right)^{\alpha+1}}dt\\
                =&\int_{0}^{T}\frac{\alpha y^{\alpha-1}\tilde{f}^{e}_{i}(t)(1+(1-p)yF_{1}(t))+\alpha y^{\alpha}(p-1)\tilde{f}^{e}_{i}(t)F_{1}(t)}{\left(1+(1-p)yF_{1}(t)\right)^{\alpha+1}}dt\\
                =&\alpha y^{\alpha-1}\int_{0}^{T}\frac{\tilde{f}^{e}_{i}(t)}{\left(1+(1-p)yF_{1}(t)\right)^{\alpha+1}}dt\\
                =&\alpha y^{\alpha-1}\mathcal{T}_{i,\alpha+1}(y).
        \end{split}
\end{equation*}
Hypothesis (H$_{1}$) implies that $F_{1}(t)$ is monotone on $[0,T]$, while hypothesis (H$_{2}$) gives the required CT-property of the integrands. 
By applying Lemma \ref{chebyshev3} with $h=y$, $g(t)=(1-p)F_1(t)$, and exponent $-(\alpha+1)\notin\mathbb N$, we conclude that $(\mathcal{T}_{0,\alpha+1}(y),\cdots,\mathcal{T}_{k,\alpha+1}(y))$ is an ECT-system on $J'\cap(0,+\infty)$.
The result then follows from Lemmas \ref{chebyshev1} and \ref{chebyshev2}.
\end{proof}

The following proposition establishes the Chebyshev property of the functions appearing in the Melnikov functions, 
which is the main ingredient in deriving the estimates in Theorems \ref{thm1} and \ref{thm2}.
\begin{proposition}\label{chebyshev4}
    Let $x_{0}(t,\rho)$ be given in \eqref{x0(t)}. Suppose that (H$_{1}$) and (H$_{2}$) hold. Then for any $k\in\mathbb{N}$ the following two ordered sets of functions
    \begin{equation}\label{chebyshevfunctions1}
        \left(1,
\int_{0}^{T} \tilde{f}^{e}_{0}(t) x_{0}^{q-p}(t,\rho)\, dt,
\cdots,
\int_{0}^{T} \tilde{f}^{e}_{k}(t) x_{0}^{q-p}(t,\rho)\, dt
\right),
    \end{equation}
and
        \begin{equation}\label{chebyshevfunctions2}
        \left(1,
\int_{0}^{T} \tilde{f}^{e}_{0}(t) x_{0}^{q-p}(t,\rho)\, dt,
\int_{0}^{T} \tilde{F}^{e}_{1}(t)\tilde{f}^{o}_{1}(t) x_{0}^{q-1}(t,\rho)\, dt,
\cdots,
\int_{0}^{T} \tilde{F}^{e}_{k}(t)\tilde{f}^{o}_{1}(t) x_{0}^{q-1}(t,\rho)\, dt
\right),
    \end{equation}
form two ECT-systems on any connected component of the set $J\cap(0,+\infty)$ and $J\cap(-\infty,0)$, where $J$ is defined by \eqref{J}.
\end{proposition}
\begin{proof}
Set $\alpha=\frac{q-p}{p-1}$.
Consider the change of variables $y=\rho^{p-1}$.
If \(\rho>0\), then
\[
y^{\alpha}\mathcal T_{i,\alpha}(y)
\big|_{y=\rho^{p-1}}
=
\int_0^T\tilde f_i^e(t)
x_0^{q-p}(t,\rho)\,dt.
\]
If \(p\) is odd and \(\rho<0\), then \(y=\rho^{p-1}>0\) and
\[
y^{\alpha}\mathcal T_{i,\alpha}(y)
\big|_{y=\rho^{p-1}}
=
(-1)^{q-p}
\int_0^T\tilde f_i^e(t)
x_0^{q-p}(t,\rho)\,dt.
\]
If \(p\) is even and \(\rho<0\), then
\[
(-y)^{\alpha}\mathcal T_{i,\alpha}(y)
\big|_{y=\rho^{p-1}}
=
(-1)^{q-p}
\int_0^T\tilde f_i^e(t)
x_0^{q-p}(t,\rho)\,dt.
\]
By Lemma \ref{964} and the fact that $(-1)^{q-p}$ is constant, \eqref{chebyshevfunctions1} forms an ECT-system on any connected component of the set $J\cap(0,+\infty)$ and $J\cap(-\infty,0)$.
Then, the fact that \eqref{chebyshevfunctions2} forms an ECT-system follows immediately from the identity for $i\geq 1$
\begin{equation}\label{dengs}
    -(q-p)\int_{0}^{T}\tilde{F}^{e}_{i}(t)\tilde{f}^{o}_{1}(t)x_{0}^{q-1}(t,\rho)dt=\int_{0}^{T} x_{0}^{q-p}(t,\rho)d\tilde{F}^{e}_{i}(t)=\int_{0}^{T} \tilde{f}^{e}_{i}(t) x_{0}^{q-p}(t,\rho)dt.
\end{equation}
\end{proof}

\section{Proof of the main results}
This section is devoted to the proof of Theorems \ref{thm1}, \ref{thm2} and \ref{thm3}.
The main idea is to express the Melnikov functions as linear combinations of functions belonging to suitable ECT-systems and then apply their zero-counting properties.
In the first part we  provide the explicit expression for the Melnikov functions of the perturbed equation \eqref{perturbation equation}.
The remaining part contains the proofs of the three main results.

\subsection{Melnikov functions for the differential equation \eqref{perturbation equation}}
We first derive the explicit expressions of $M_{1}$ and $M_{2}$, and identify the function spaces to which they belong.
\begin{proposition}
    Suppose that (H$_{1}$) and (H$_{2}$) hold. Then the following statements hold for the perturbed equation \eqref{perturbation equation}:
\begin{itemize}
    \item[(i)] The first-order Melnikov function is given by 
    \begin{equation*}
            M_{1}(\rho;\mu)=a_{10}\int_{-T}^{T}f_{0}(t)\,dt\rho^{p}+\sum_{k\in\mathcal{N}_{m}^{\text{non-odd}}}\left(2b_{1k}\rho^{p}\int_{0}^{T}f^{e}_{k}(t)x_{0}(t,\rho)^{q-p}dt\right).
    \end{equation*}
    Moreover, $M_{1}(\rho;\mu)\equiv 0$ if and only if $a_{10}=b_{1k}=0$ for all $k\in\mathcal{N}_{m}^{\text{non-odd}}$.
    \item[(ii)] If $M_{1}(\rho;\mu)\equiv 0$, then the second-order Melnikov function is given by $M_{2}(\rho;\mu)=M_{21}(\rho;\mu)+(q-p)M_{22}(\rho;\mu)$, where
    \begin{equation*}
        \begin{split}
            M_{21}(\rho;\mu)=&a_{20}\int_{-T}^{T}f_{0}(t)\,dt\rho^{p}+2b_{20}\rho^{p}\int_{0}^{T}f^{e}_{0}(t)x_{0}^{q-p}(t,\rho)dt\\
            &-(q-p)\sum_{i\in\mathcal{N}_{m}^{\text{non-odd}}\setminus\{0\}}\left(2b_{2i}\rho^{p}\int_{0}^{T}F^{e}_{i}(t)f_{1}(t)x_{0}^{q-1}(t,\rho)dt\right).
        \end{split}
    \end{equation*}
    and
    \begin{equation*}
        \begin{split}
            M_{22}(\rho;\mu)=&\sum_{\substack{l\in\mathcal{N}^{\text{non-odd}}_{n}\setminus\{0\}\\ k\in\mathcal{N}^{\text{odd}}_{m}}}2a_{1l}b_{1k}\rho^{p}\int_{0}^{T}F^{e}_{l}(t)f^{o}_{k}(t)x_{0}(t,\rho)^{q-1}dt.
        \end{split}
    \end{equation*}
\end{itemize}
\end{proposition}
\begin{proof}
    (i) The expression of $M_{1}$ follows from Lemma \ref{Melnikovfunction} and we have 
\begin{equation}\label{M1}
    \begin{split}
        M_{1}(\rho;\mu)
        &=\rho^{p}\int_{-T}^{T}P_{1}(t)+Q_{1}(t)x_{0}(t,\rho)^{q-p}dt\\
        &=\int_{-T}^{T}\left(\sum_{k=0}^{n}a_{1k}f_{k}(t)\right)dt\rho^{p}+\int_{-T}^{T}\left(\sum_{k=0}^{m}b_{1k}\rho^{p}\left(f^{e}_{k}(t)+f^{o}_{k}(t)\right)x_{0}(t,\rho)^{q-p}\right)dt\\
        &=a_{10}\int_{-T}^{T}f_{0}(t)\,dt\rho^{p}+\sum_{k=0}^{m}\left(b_{1k}\rho^{p}\int_{-T}^{T}f^{e}_{k}(t)x_{0}(t,\rho)^{q-p}\,dt\right)\\
        &=a_{10}\int_{-T}^{T}f_{0}(t)\,dt\rho^{p}+\sum_{k\in\mathcal{N}_{m}^{\text{non-odd}}}\left(2b_{1k}\rho^{p}\int_{0}^{T}f^{e}_{k}(t)x_{0}(t,\rho)^{q-p}\,dt\right),
    \end{split}
\end{equation}
where we have used the fact that $\int_{-T}^{T} f_{k}(t) dt = 0$ for all $k\geq 1$ and the solution $x_0(t, \rho)$ is even in $t$.
The second assertion in (i) follows by Proposition \ref{chebyshev4}, which shows that 
$\left(\int_{0}^{T} f^{e}_{k}(t) x_{0}^{q-p}(t,\rho)\, dt\right)_{k\in\mathcal{N}_{m}^{\text{non-odd}}}$ is an ECT-system in $J\cap(0,+\infty)$ and $J\cap(-\infty,0)$.

    (ii) 
Note that $M_{1}(\rho;\mu)\equiv0$ if and only if $a_{10}=b_{1k}=0, k\in\mathcal{N}_{m}^{\text{non-odd}}$.
According to Lemma \ref{Melnikovfunction}, we have
\begin{equation*}
    \begin{split}
        M_{2}(\rho;\mu)&=\rho^{p}\int_{-T}^{T}P_{2}(t)+Q_{2}(t)x_{0}(t,\rho)^{q-p}+(q-p)Q_{1}(t)S(t,\rho)x_{0}(t,\rho)^{q-1}dt\\
                &:=M_{21}+(q-p)M_{22}
    \end{split}
\end{equation*}
where
\begin{equation*}
        S(t,\rho)=\int_{-T}^{t}\left(P_{1}(s)+Q_{1}(s)x_{0}(s,\rho)^{q-p}\right)ds,
\end{equation*}
\begin{equation}\label{M21}
    \begin{split}
        M_{21}=&\rho^{p}\int_{-T}^{T}P_{2}(t)+Q_{2}(t)x_{0}(t,\rho)^{q-p}\,dt\\
        =&\sum_{k=0}^{n}a_{2k}\int_{-T}^{T}f_{k}(t)dt\rho^{p}+\sum_{k=0}^{m}\left(2b_{2k}\rho^{p}\int_{0}^{T}f^{e}_{k}(t)x_{0}^{q-p}(t,\rho)dt\right)\\
            =&a_{20}\int_{-T}^{T}f_{0}(t)\,dt\rho^{p}+\sum_{i\in\mathcal{N}_{m}^{\text{non-odd}}}\left(2b_{2i}\rho^{p}\int_{0}^{T}f^{e}_{i}(t)x_{0}^{q-p}(t,\rho)dt\right)\\
            =&a_{20}\int_{-T}^{T}f_{0}(t)\,dt\rho^{p}+2b_{20}\rho^{p}\int_{0}^{T}f^{e}_{0}(t)x_{0}^{q-p}(t,\rho)dt\\
            &-(q-p)\sum_{i\in\mathcal{N}_{m}^{\text{non-odd}}\setminus\{0\}}\left(2b_{2i}\rho^{p}\int_{0}^{T}F^{e}_{i}(t)f^{o}_{1}(t)x_{0}^{q-1}(t,\rho)dt\right),
            \end{split}
\end{equation}
and
\begin{equation}\label{M22}
    \begin{split}
        M_{22}&=\rho^{p}\int_{-T}^{T}Q_{1}(t)S(t,\rho)x_{0}(t,\rho)^{q-1}dt\\
            &=\rho^{p}\int_{-T}^{T}\left(\sum_{k\in\mathcal{N}^{\text{odd}}_{m}}b_{1k}f_{k}(t)\right)\left(\int_{-T}^{t}\left(P_{1}(s)+Q_{1}(s)x_{0}(s,\rho)^{q-p}\right)ds\right) x_{0}(t,\rho)^{q-1}dt\\
            &=\rho^{p}\int_{-T}^{T}\left(\sum_{k\in\mathcal{N}^{\text{odd}}_{m}}b_{1k}f_{k}(t)\right)\left(\int_{-T}^{t}P_{1}(s)ds\right) x_{0}(t,\rho)^{q-1}dt\\
            &=\rho^{p}\int_{-T}^{T}\left(\sum_{k\in\mathcal{N}^{\text{odd}}_{m}}b_{1k}f_{k}(t)\right)\left(\sum_{l=1}^{n}a_{1l}\int_{-T}^{t}f_{l}(s)ds\right) x_{0}(t,\rho)^{q-1}dt\\
            &=\rho^{p}\int_{-T}^{T}\left(\sum_{k\in\mathcal{N}^{\text{odd}}_{m}}b_{1k}f^{o}_{k}(t)\right)\left(\sum_{l=1}^{n}a_{1l}\int_{-T}^{t}f^{e}_{l}(s)ds\right) x_{0}(t,\rho)^{q-1}dt\\
            &=\rho^{p}\int_{-T}^{T}\left(\sum_{k\in\mathcal{N}^{\text{odd}}_{m}}b_{1k}f^{o}_{k}(t)\right)\left(\sum_{l\in\mathcal{N}^{\text{non-odd}}_{n}\setminus\{0\}}a_{1l}\int_{-T}^{t}f^{e}_{l}(s)ds\right) x_{0}(t,\rho)^{q-1}dt\\
            &=\sum_{\substack{l\in\mathcal{N}^{\text{non-odd}}_{n}\setminus\{0\} \\ k\in\mathcal{N}^{\text{odd}}_{m}}}a_{1l}b_{1k}\rho^{p}\int_{-T}^{T}F^{e}_{l}(t)f^{o}_{k}(t)x_{0}(t,\rho)^{q-1}dt\\
            &=\sum_{\substack{l\in\mathcal{N}^{\text{non-odd}}_{n}\setminus\{0\} \\ k\in\mathcal{N}^{\text{odd}}_{m}}}2a_{1l}b_{1k}\rho^{p}\int_{0}^{T}F^{e}_{l}(t)f^{o}_{k}(t)x_{0}(t,\rho)^{q-1}dt.
    \end{split}
\end{equation}
\end{proof}
\begin{remark}\label{expression of MF}
    For the zero estimates below, it is convenient to rewrite the expressions of the Melnikov functions obtained in the above proposition using the setting in Section \ref{setting}:
        \begin{equation}
        \begin{split}
            M_{1}(\rho;\mu):=&a_{10}\int_{-T}^{T}f_{0}(t)\,dt\rho^{p}+\sum_{k=0}^{m_{1}}\left(2\tilde{b}_{1,k}\rho^{p}\int_{0}^{T}\tilde{f}^{e}_{k}(t)x_{0}(t,\rho)^{q-p}dt\right),\\
            M_{21}(\rho;\mu):=&a_{20}\int_{-T}^{T}f_{0}(t)\,dt\rho^{p}+2\tilde{b}_{20}\rho^{p}\int_{0}^{T}\tilde{f}^{e}_{0}(t)x_{0}(t,\rho)^{q-p}dt\\
            &-(q-p)\sum_{i=1}^{m_{1}}\left(2\tilde{b}_{2i}\rho^{p}\int_{0}^{T}\tilde{F}^{e}_{i}(t)\tilde{f}^{o}_{1}(t)x_{0}^{q-1}(t,\rho)dt\right),\\
            M_{22}(\rho;\mu):=&\sum_{\substack{1\leq k\leq m-m_{1} \\ 1\leq l\leq n_{1}}}2\tilde{a}_{1,l}\tilde{b}_{1,k}\rho^{p}\int_{0}^{T}\tilde{F}^{e}_{l}(t)\tilde{f}^{o}_{k}(t)x_{0}(t,\rho)^{q-1}dt.
        \end{split}
    \end{equation}
    Moreover, $M_{1}(\rho;\mu)\equiv 0$ if and only if $a_{10}=\tilde{b}_{1,k}=0$ for all $k=0,1,\cdots,m_{1}$.
\end{remark}

The next proposition shows that, under (H$_{3}$), the second-order Melnikov function also belongs to a finite-dimensional space with controlled Chebyshev structure.
\begin{proposition}\label{proposition1}
    Suppose that (H$_{1}$), (H$_{2}$) and (H$_{3}$) hold.
    Setting $K=\max\{m-m_{1}+n_{1}, m_{1}+1\}$ and $\mathcal{A}=\{\mu\in\mathbb{R}^{2(n+m+2)}\mid M_{1}(\rho;\mu)=0 \text{ for all }\rho\in J\}$, 
    there exists a surjective map $\beta:\mathcal{A}\rightarrow\mathbb{R}^{K+1}$ such that
    \begin{equation*}
        \begin{split}
        M_{2}(\rho;\mu)=&\beta_{1}(\mu)\rho^{p}+\beta_{2}(\mu)\rho^{p}\int_{0}^{T}\tilde{f}^{e}_{0}(t)x_{0}(t,\rho)^{q-p}\,dt\\
        &+\sum_{i=1}^{K-1}\beta_{i+2}(\mu)\rho^{p}\int_{0}^{T}\tilde{F}^{e}_{i}(t)\tilde{f}^{o}_{1}(t)x_{0}(t,\rho)^{q-1}\,dt.
        \end{split}
    \end{equation*}
\end{proposition}
\begin{proof}
    By Remark \ref{expression of MF}, we have that 
    \begin{equation*}
        \begin{split}
            M_{2}(\rho;\mu)=&M_{21}(\rho;\mu)+(q-p)M_{22}(\rho;\mu)\\
            =&a_{20}\int_{-T}^{T}f_{0}(t)\,dt\rho^{p}+2\tilde{b}_{20}\rho^{p}\int_{0}^{T}\tilde{f}^{e}_{0}(t)x_{0}(t,\rho)^{q-p}dt\\
            &-(q-p)\sum_{i=1}^{m_{1}}\left(2\tilde{b}_{2i}\rho^{p}\int_{0}^{T}\tilde{F}^{e}_{i}(t)\tilde{f}^{o}_{1}(t)x_{0}^{q-1}(t,\rho)dt\right)\\
            &+(q-p)\sum_{\substack{1\leq k\leq m-m_{1} \\ 1\leq l\leq n_{1}}}\left(2\tilde{a}_{1,l}\tilde{b}_{1,k}\rho^{p}\int_{0}^{T}\tilde{F}^{e}_{l}(t)\tilde{f}^{o}_{k}(t)x_{0}(t,\rho)^{q-1}dt\right).
        \end{split}
    \end{equation*}
    Under hypothesis (H$_{{3}}$), we know that 
    \begin{equation*}
        M_{2}(\rho;\mu)\in \operatorname{span}\left(\left\{\rho^{p}, \rho^{p}\int_{0}^{T}\tilde{f}^{e}_{0}(t)x_{0}(t,\rho)^{q-p}\,dt, \cdots,\rho^{p}\int_{0}^{T}\tilde{F}^{e}_{K-1}(t)\tilde{f}^{o}_{1}(t)x_{0}(t,\rho)^{q-1}\,dt\right\}\right),
    \end{equation*}
which has dimension $K+1$.
Hence there exists a unique $\hat{\beta}=\hat{\beta}(\mu)\in \mathbb{R}^{K+1}$ such that
    \begin{equation*}
        M_{2}(\rho;\mu)=\hat{\beta}_{1}\rho^{p}+\hat{\beta}_{2}\rho^{p}\int_{0}^{T}\tilde{f}^{e}_{0}(t)x_{0}(t,\rho)^{q-p}\,dt+\rho^{p}\sum_{i=1}^{K-1}\hat{\beta}_{i+2}\int_{0}^{T}\tilde{F}^{e}_{i}(t)\tilde{f}^{o}_{1}(t)x_{0}(t,\rho)^{q-1}\,dt.
    \end{equation*}
Next, we prove that $\hat{\beta}:\mathcal{A}\rightarrow \mathbb{R}^{K+1}$ is a surjective map.
It suffices to construct a set $\mathcal{C}$ such that $\mathbb{R}^{K+1}=\hat{\beta}\left(\mathcal{A}\cap\mathcal{C}\right)$.
{
Before constructing the set \(\mathcal{C}\), we claim that 
\begin{equation}\label{calim1}
    \begin{split}
\mathcal{B}_{k+l-1}
=&
\operatorname{span}
\left(
\left\{
\tilde{F}_{i}^{e}\tilde{f}_{1}^{o}
\right\}_{1\leq i\leq l-1}
\bigcup
\left\{
\tilde{F}_{l}^{e}\tilde{f}_{j}^{o}
\right\}_{1\leq j\leq k}
\right)\\
=&
\operatorname{span}
\left(
\left\{
\tilde{F}_{i}^{e}\tilde{f}_{1}^{o}
\right\}_{1\leq i\leq l-1}
\bigcup
\left\{
\tilde{F}_{j}^{e}\tilde{f}_{l}^{o}
\right\}_{1\leq j\leq k}
\right).
    \end{split}    
\end{equation}
Recall that $\mathcal{B}_{r}=\operatorname{span}\big(\big\{\tilde{F}^{e}_{1}\tilde{f}^{o}_{1},\cdots,\tilde{F}^{e}_{r}\tilde{f}^{o}_{1}\big\}\big)$.
We first prove the first equality in \eqref{calim1} by induction on $k$. 
When \(k=1\), we obtain that 
$$\mathcal{B}_{l}
=
\operatorname{span}
\left(
\left\{
    \tilde{F}_{i}^{e}\tilde{f}_{1}^{o}
\right\}_{1\leq i\leq l}
\right),$$ which is clearly true by the definition of $\mathcal{B}_{r}$.
Then we assume that the first equality in \eqref{calim1} holds for $k$ and it remains only to prove it for $k+1$, that is,
\begin{equation}\label{calim2}
\mathcal{B}_{k+l}
=
\operatorname{span}
\left(
\left\{
\tilde{F}_{i}^{e}\tilde{f}_{1}^{o}
\right\}_{1\leq i\leq l-1}
\bigcup
\left\{
\tilde{F}_{l}^{e}\tilde{f}_{j}^{o}
\right\}_{1\leq j\leq k+1}
\right).    
\end{equation}
By the first equality in \eqref{calim1}, we have 
\begin{equation}\label{calim3}
    \tilde{F}_{m}^{e}\tilde{f}_{1}^{o}\in 
\operatorname{span}
\left(
\left\{
\tilde{F}_{i}^{e}\tilde{f}_{1}^{o}
\right\}_{1\leq i\leq l-1}
\bigcup
\left\{
\tilde{F}_{l}^{e}\tilde{f}_{j}^{o}
\right\}_{1\leq j\leq k+1}
\right),\ 1\leq m\leq k+l-1.
\end{equation}
By the first part of hypothesis (H$_{3}$), we know that 
$$\mathcal{B}_{k+l}
\supseteq
\operatorname{span}
\left(
\left\{
\tilde{F}_{i}^{e}\tilde{f}_{1}^{o}
\right\}_{1\leq i\leq l-1}
\bigcup
\left\{
\tilde{F}_{l}^{e}\tilde{f}_{j}^{o}
\right\}_{1\leq j\leq k+1}
\right).
$$
Therefore, we only need to prove that 
\begin{equation}\label{aim}
    \tilde{F}_{k+l}^{e}\tilde{f}_{1}^{o}\in 
\operatorname{span}
\left(
\left\{
\tilde{F}_{i}^{e}\tilde{f}_{1}^{o}
\right\}_{1\leq i\leq l-1}
\bigcup
\left\{
\tilde{F}_{l}^{e}\tilde{f}_{j}^{o}
\right\}_{1\leq j\leq k+1}
\right).
\end{equation}
By the first part of hypothesis (H$_{3}$), we have 
$\tilde{F}_{l}^{e}\tilde{f}_{k+1}^{o}\in\mathcal{B}_{k+l}\setminus\mathcal{B}_{k+l-1}$, and hence 
\begin{equation*}
   \tilde{F}_{l}^{e}\tilde{f}_{k+1}^{o}=a_{k+l}\tilde{F}_{k+l}^{e}\tilde{f}_{1}^{o}+a_{k+l-1}\tilde{F}_{k+l-1}^{e}\tilde{f}_{1}^{o}+\dots+a_{1}\tilde{F}_{1}^{e}\tilde{f}_{1}^{o}, 
\end{equation*}
where $a_{k+l}\not=0$.
It follows immediately from the above equality and \eqref{calim3} that \eqref{aim} holds.
The second equality in \eqref{calim1} follows by the same induction argument on $l$, and hence the proof is omitted.
Therefore, we have verified this claim.
}

Then, by the second part of hypothesis (H$_{3}$), we consider the following {three} cases: 

Case 1: $m-m_{1}+n_{1}\geq m_{1}+1$ {and $m_{1}\geq m-m_{1}-1$}.  
We define
{
\begin{equation*}
    \mathcal{C}=\{\mu\in\mathbb{R}^{2(n+m+2)}\mid \tilde{b}_{1,m-m_{1}}=1, \tilde{b}_{1,k}=0 \text{ for } 1\leq k < m-m_{1}, \tilde{b}_{2,i}=0 \text{ for } m-m_{1}\leq i \leq m_{1}\}.
\end{equation*}
}
If $\mu\in \mathcal{A}\cap\mathcal{C}$, we have that
{
    \begin{equation*}
        \begin{split}
            M_{2}(\rho;\mu)
            =&a_{20}\int_{-T}^{T}f_{0}(t)\,dt\rho^{p}+2\tilde{b}_{20}\rho^{p}\int_{0}^{T}\tilde{f}^{e}_{0}(t)x_{0}^{q-p}(t,\rho)dt\\
            &-(q-p)\sum_{i=1}^{m-m_{1}-1}\left(2\tilde{b}_{2i}\rho^{p}\int_{0}^{T}\tilde{F}^{e}_{i}(t)\tilde{f}^{o}_{1}(t)x_{0}^{q-1}(t,\rho)dt\right)\\
            &+(q-p)\sum_{l=1}^{n_{1}}\left(2\tilde{a}_{1,l}\rho^{p}\int_{0}^{T}\tilde{F}^{e}_{l}(t)\tilde{f}^{o}_{m-m_{1}}(t)x_{0}(t,\rho)^{q-1}dt\right).
        \end{split}
\end{equation*}
}
According to \eqref{calim1}, we know that $\mathbb{R}^{K+1}=\hat{\beta}\left(\mathcal{A}\cap\mathcal{C}\right)$.

{
    Case 2: $m-m_{1}+n_{1}\geq m_{1}+1$ {and $m_{1}\geq n_{1}-1$}.  
We define
{
\begin{equation*}
    \mathcal{C}=\{\mu\in\mathbb{R}^{2(n+m+2)}\mid \tilde{a}_{1,n_{1}}=1, \tilde{a}_{1,k}=0 \text{ for } 1\leq k < n_{1}, \tilde{b}_{2,i}=0 \text{ for } n_{1}\leq i \leq m_{1}\}.
\end{equation*}
}
If $\mu\in \mathcal{A}\cap\mathcal{C}$, we have that
{
    \begin{equation*}
        \begin{split}
            M_{2}(\rho;\mu)
            =&a_{20}\int_{-T}^{T}f_{0}(t)\,dt\rho^{p}+2\tilde{b}_{20}\rho^{p}\int_{0}^{T}\tilde{f}^{e}_{0}(t)x_{0}^{q-p}(t,\rho)dt\\
            &-(q-p)\sum_{i=1}^{n_{1}-1}\left(2\tilde{b}_{2i}\rho^{p}\int_{0}^{T}\tilde{F}^{e}_{i}(t)\tilde{f}^{o}_{1}(t)x_{0}^{q-1}(t,\rho)dt\right)\\
            &+(q-p)\sum_{l=1}^{m-m_{1}}\left(2\tilde{b}_{1,l}\rho^{p}\int_{0}^{T}\tilde{F}^{e}_{n_{1}}(t)\tilde{f}^{o}_{l}(t)x_{0}(t,\rho)^{q-1}dt\right).
        \end{split}
\end{equation*}
}
According to \eqref{calim1}, we know that $\mathbb{R}^{K+1}=\hat{\beta}\left(\mathcal{A}\cap\mathcal{C}\right)$.
}

Case 3: $m-m_{1}+n_{1}< m_{1}+1$. 
We define
\begin{equation*}
    \mathcal{C}=\{\mu\in\mathbb{R}^{2(n+m+2)}\mid \tilde{b}_{1,k}=0 \text{ for } 1\leq k \leq m-m_{1}\}.
\end{equation*}
If $\mu\in \mathcal{A}\cap\mathcal{C}$, we have that
\begin{equation*}
    \begin{split}
            M_{2}(\rho;\mu)
            =&a_{20}\int_{-T}^{T}f_{0}(t)\,dt\rho^{p}+2\tilde{b}_{20}\rho^{p}\int_{0}^{T}\tilde{f}^{e}_{0}(t)x_{0}^{q-p}(t,\rho)dt\\
            &-(q-p)\sum_{i=1}^{m_{1}}\left(2\tilde{b}_{2i}\rho^{p}\int_{0}^{T}\tilde{F}^{e}_{i}(t)\tilde{f}^{o}_{1}(t)x_{0}^{q-1}(t,\rho)dt\right)
    \end{split}
\end{equation*}
It is clear that $\mathbb{R}^{K+1}=\hat{\beta}\left(\mathcal{A}\cap\mathcal{C}\right)$.

Finally, we define $\beta:=\widehat\beta|_{\mathcal A}$ which is exactly the surjection from $\mathcal{A}$ to $\mathbb{R}^{K+1}$ we have been seeking.

This completes the proof.
\end{proof}

\subsection{Proofs of the main results}
This part is devoted to the proof of Theorems \ref{thm1}, \ref{thm2} and \ref{thm3}.
First, we give the proof of Theorem \ref{thm1}.

\begin{proof}[Proof of Theorem \ref{thm1}]
By applying Remark \ref{expression of MF}, we have that
\begin{equation*}
    M_{1}(\rho;\mu)=\rho^{p} \tilde{M}_{1}(\rho;\mu) \text{ with } \tilde{M}_{1}(\rho;\mu)=\tilde{a}_{1,0}+\sum_{k=0}^{m_{1}}\left(2\tilde{b}_{1,k}\int_{0}^{T}\tilde{f}^{e}_{k}(t)x_{0}(t,\rho)^{q-p}dt\right),
\end{equation*}
where $\tilde{a}_{1,0}:=a_{10}\int_{-T}^{T}f_{0}(t)\,dt$.
It is clear that $M_{1}$ and $\tilde{M}_{1}$ have the same number of positive and negative zeros, counted with multiplicities.
For convenience, we introduce the function
\begin{equation*}
    \begin{split}
        N(\rho;\mu):=\sum_{k=0}^{m_{1}}\left(2\tilde{b}_{1,k}\int_{0}^{T}\tilde{f}^{e}_{k}(t)x_{0}(t,\rho)^{q-p}dt\right)=\sum_{k=0}^{m_{1}}\left(2\tilde{b}_{1,k}\rho^{q-p}\mathcal{T}_{k,\alpha}(\rho^{p-1})\right),
    \end{split}
\end{equation*}
where $\alpha=\frac{q-p}{p-1}$ and the definition of $\mathcal{T}_{k,\alpha}$ is given in \eqref{Definition of T}.
Differentiating $N(\rho;\mu)$ with respect to $\rho$, we obtain
\begin{equation*}
    \begin{split}
        N'(\rho;\mu)=(q-p)\rho^{q-p-1}\sum_{k=0}^{m_{1}}\left(2\tilde{b}_{1,k}\mathcal{T}_{k,\alpha+1}(\rho^{p-1})\right).
    \end{split}
\end{equation*}
Moreover, applying Lemma \ref{chebyshev3} to $N$ and $N'$, the Chebyshev property gives the following estimates.

\medskip
\noindent
{Claim 1.}
Provided that $N(\cdot;\mu)\not\equiv 0$, the function $N(\cdot;\mu)$ has at most $m_1$ zeros in $J\cap(-\infty,0)$, counted with multiplicities. The same upper bound holds in $J\cap(0,+\infty)$.

\medskip
\noindent
{Claim 2.}
Provided that $N'(\cdot;\mu)\not\equiv 0$, the function $N'(\cdot;\mu)$ has at most $m_1$ zeros in $J\cap(-\infty,0)$, counted with multiplicities. The same upper bound holds in $J\cap(0,+\infty)$.

\medskip
\noindent
{Claim 3.}
Suppose that $p$ is even. If $N(\cdot;\mu)\not\equiv 0$, then $N(\cdot;\mu)$ has at most $m_1$ zeros in $J\setminus\{0\}$, counted with multiplicities. Likewise, if $N'(\cdot;\mu)\not\equiv 0$, then $N'(\cdot;\mu)$ has at most $m_1$ zeros in $J\setminus\{0\}$, counted with multiplicities.

The degenerate cases \(N(\cdot;\mu)\equiv0\) and \(N'(\cdot;\mu)\equiv0\) do not affect the conclusions. 
Indeed, in either case \(N(\cdot;\mu)\) is constant on each interval under consideration, and hence so is
\[
\tilde M_1(\rho;\mu)
=
\tilde a_{1,0}+N(\rho;\mu).
\]
Since \(M_1(\cdot;\mu)\not\equiv0\), this constant is nonzero. 
In either degenerate case, $\tilde M_1(\cdot;\mu)$ has no zeros on the interval under consideration, and hence these cases do not affect the subsequent zero-counting arguments.

Let us prove next each one of the assertions in the statement of the result:

\begin{itemize}
    \item[(i)] From the expression of $\tilde{M}_{1}(\rho;\mu)$ and Proposition \ref{chebyshev4}, we obtain $\mathbb{Z}(M_{1})^{\pm}\leq m_{1}+1$.
Moreover, there exists $\mu_{0}\in\mathbb{R}^{2(n+m+2)}$ such that $\tilde{M}_{1}(\rho;\mu_{0})$ has exactly $m_{1}+1$ positive simple zeros.
Since $\tilde{M}_{1}(\rho;\mu_{0})$ is an even function with respect to $\rho$ in the case that both $p$ and $q$ are odd, 
it also has $m_{1}+1$ negative simple zeros.
Therefore, $\tilde{M}_{1}(\rho;\mu_{0})$ has $2(m_{1}+1)$ simple zeros in $J\setminus\{0\}$.

\item[(ii)] Proposition~\ref{chebyshev4} again gives $\mathbb{Z}(M_{1})^{\pm}\leq m_{1}+1$.
We prove that these two equalities cannot hold simultaneously. 
Suppose by contradiction that there exists $\mu_{0}\in\mathbb{R}^{2(n+m+2)}$ such that $\tilde{M}_{1}(\rho;\mu_{0})$ has $m_{1}+1$ positive zeros $0<\rho_{1}^{+}\leq\cdots\leq\rho_{m_{1}+1}^{+}$ and $m_{1}+1$ negative zeros $0>-\rho_{1}^{-}\geq\cdots\geq-\rho_{m_{1}+1}^{-}$, counted with multiplicities.
Notice that $N(\rho;\mu_{0})$ is odd, since $q-p$ is odd and $p-1$ is even. 
Therefore,
\begin{equation}\label{21}
    \begin{split}
        &N(\rho_i^+;\mu_{0}) = \tilde{M}_1(\rho_i^+;\mu_{0}) - \tilde{a}_{1,0} =-\tilde{a}_{1,0},\\
        &N(\rho_j^-;\mu_{0}) = -\tilde{M}_1(-\rho_j^-;\mu_{0}) + \tilde{a}_{1,0} = \tilde{a}_{1,0},
    \end{split}
\end{equation}
for all $i, j \in \{1, 2, \dots, m_{1}+1\}$. 
This shows in particular that $\tilde{a}_{1,0} \neq 0$ (otherwise we get a contradiction with Claim 1), which in turn implies that $\rho_i^+ \neq \rho_j^-$ for all $i$ and $j$. 
On the other hand, $N' = \tilde{M}_1'$ has $m_{1}$ positive zeros $0 < \varrho_1^+ \le \dots \le \varrho_{m_{1}}^+$, and $m_{1}$ negative zeros $0 > -\varrho_1^- \ge \dots \ge -\varrho_{m_{1}}^-$, counted with multiplicities, satisfying $\varrho_i^\pm \in [\rho_i^\pm, \rho_{i+1}^\pm]$ for all $i$. 
Since $N'$ is even, Claim 2 implies that $\varrho_i^+ = \varrho_i^-$ for all $i$. 
Hence,
\begin{equation}\label{22}
    \quad [\rho_i^+, \rho_{i+1}^+] \cap [\rho_i^-, \rho_{i+1}^-] \neq \emptyset \quad \text{for all } i.
\end{equation}
For each $i = 1, 2, \dots, m_{1}+1$, let $J_i$ denote the open interval with endpoints $\rho_i^+$ and $\rho_i^-$. 
It follows from \eqref{22} that the intervals $J_1, J_2, \dots, J_{m_{1}+1}$ are pairwise non-intersecting. 
Moreover, by \eqref{21}, each $J_i$ contains at least one zero of $N$, which contradicts Claim 1. 
Therefore, either $\mathbb{Z}(M_{1})^{+} < m_{1}+1$ or $\mathbb{Z}(M_{1})^{-} < m_{1}+1$.

\begin{itemize}
    \item[(ii.1)] By Lemma \ref{chebyshev3}, $\left(\rho^{q-p}\mathcal{T}_{0,\alpha}(\rho^{p-1}), \rho^{q-p}\mathcal{T}_{1,\alpha}(\rho^{p-1}), \dots, \rho^{q-p}\mathcal{T}_{m_{1},\alpha}(\rho^{p-1})\right)$ is an ECT-system in $J\cap(0, +\infty)$. 
Consequently, there exist coefficients $\tilde{b}_{1,0}, \dots, \tilde{b}_{1,m_{1}} \in \mathbb{R}$ such that
$$\rho \longmapsto \rho^{q-p} \sum_{k=0}^{m_{1}} 2\tilde{b}_{1,k} \mathcal{T}_{k,\alpha}(\rho^{p-1})$$
has $m_{1}$ positive simple zeros and $m_{1}$ negative simple zeros. 
Here we have used the fact that this function is odd under the parity assumptions on $p$ and $q$. 
For the same reason, $\rho = 0$ is a zero with odd multiplicity. 
Choosing $\tilde{a}_{1,0}$ sufficiently small, the function
$$\rho \longmapsto \tilde{a}_{1,0} + \rho^{q-p} \sum_{k=0}^{m_{1}} 2\tilde{b}_{1,k} \mathcal{T}_{k,\alpha}(\rho^{p-1})$$
has $2m_{1}+1$ simple zeros in $J \setminus \{0\}$. 
Therefore, there exists $\mu_0 \in \mathbb{R}^{2(n+m+2)}$ such that $M_1(\rho; \mu_0)$ has $2m_{1}+1$ simple zeros in $J \setminus \{0\}$.

\item[(ii.2)] 
Using the same argument as above and choosing $\tilde a_{10}=0$, we obtain the existence of $\mu_0\in\mathbb R^{2(n+m+2)}$ such that $M_1(\rho;\mu_0)$ has $2m_1$ simple zeros in $J\setminus\{0\}$.
\end{itemize}

\item[(iii)] 
Let $\rho_{\mathbb{Z}(M_{1})^{-}}^-\leq\cdots\leq\rho_{1}^-<0<\rho_1^+\leq\cdots\leq\rho_{\mathbb{Z}(M_{1})^{+}}^+$ be the negative and positive zeros of $\tilde{M}_1(\rho; \mu)$ in $J$, respectively.
By Rolle's theorem, $N'=\tilde{M}_1'$ has at least $\mathbb{Z}(M_{1})^{+} - 1$ zeros in $J\cap[\rho_1^+, +\infty)$ and at least $\mathbb{Z}(M_{1})^{-} - 1$ zeros in $J\cap(-\infty, \rho_{1}^-]$, counted with multiplicities. 
Write
\begin{equation*}
   N'(\rho;\mu)=(q-p)\rho^{q-p-1}\mathcal E(\rho^{p-1};\mu)\text{ with } \mathcal E(y;\mu):=\sum_{k=0}^{m_1}2\tilde b_{1,k}\mathcal T_{k,\alpha+1}(y). 
\end{equation*}
Since $p$ is even, the map $\rho\mapsto y=\rho^{p-1}$ is strictly increasing.
Therefore, $N'$ and $\mathcal{E}$ have the same positive and negative zeros.
If $\mathcal E(0;\mu)=0$, then, in addition to the zeros obtained on the two half-axes, $y=0$ is a zero of $\mathcal E$. 
Hence $\mathcal E$ has at least $\mathbb{Z}(M_{1})^{+}+\mathbb{Z}(M_{1})^{-}-1$ zeros on $J'$, counted with multiplicities.
Assume now that $\mathcal E(0;\mu)\neq0$. 
Since $q-p-1$ is even, $N'$ has a constant sign on a sufficiently small punctured neighborhood of the origin. 
On the other hand,
\begin{equation}\label{86g}
    N(\rho_{1}^-;\mu)=N(\rho_1^+;\mu)=-\tilde a_{1,0}.
\end{equation}
This shows the existence of a zero of $N'$ on $(\rho_{1}^-,\rho_1^+)\setminus\{0\}$.
Otherwise, since $q-p-1$ is even and $\mathcal E(0;\mu)\neq0$, the derivative $N'(\cdot;\mu)$ would have the same strict sign on $(\rho_1^-,0)$ and $(0,\rho_1^+)$. Hence $N(\cdot;\mu)$ would be strictly monotone on $(\rho_1^-,\rho_1^+)$, contradicting \eqref{86g}.
Consequently, $\mathcal E$ again has at least $\mathbb{Z}(M_{1})^{+}+\mathbb{Z}(M_{1})^{-}-1$ zeros on $J'$, counted with multiplicities.
Since $(\mathcal T_{0,\alpha+1},\ldots,\mathcal T_{m_1,\alpha+1})$ is an ECT-system on $J'$, we obtain
$\mathbb{Z}(M_{1})^{+}+\mathbb{Z}(M_{1})^{-}-1\leq m_1$, and hence $\mathbb{Z}(M_{1})^{+}+\mathbb{Z}(M_{1})^{-}\leq m_1+1$.

Repeating the argument used in the previous cases, there exists $\mu_0 \in \mathbb{R}^{2(n+m+2)}$ such that $M_1(\cdot; \mu_0)$ has $m_{1}+1$ simple zeros on the interval $J\cap(0,+\infty)$.

\item[(iv)] The function $N' = \tilde{M}_1'$ has at least $\mathbb{Z}(M_{1})^{+} + \mathbb{Z}(M_{1})^{-} - 2$ zeros in $J\setminus \{0\}$ counted with multiplicities. 
Therefore, by Claim 3, we obtain $\mathbb{Z}(M_{1})^{+} + \mathbb{Z}(M_{1})^{-} \le m_{1}+2$. 
To prove that this upper bound is sharp, we use the fact that, by Lemma \ref{chebyshev3}, 
$(\mathcal{T}_{0,\alpha+1}, \mathcal{T}_{1,\alpha+1}, \dots, \mathcal{T}_{m_{1},\alpha+1})$ is an ECT-system on $J$.
We distinguish two cases:

Case 1: $q$ even and $p < q$. 
Choose $\beta_{2}, \dots, \beta_{m_{1}+2} \in \mathbb{R}$ such that
\begin{equation}\label{F}
    \quad \mathcal{G}(\rho) := \sum_{k=0}^{m_{1}} \beta_{k+2} \mathcal{T}_{k,\alpha}(\rho^{p-1})
\end{equation}
has $m_{1}$ simple zeros $0 < \rho_1 < \rho_2 < \dots < \rho_{m_{1}}$ and satisfies $\mathcal{G}(0) \neq 0$. 
Hence, $\rho^{q-p}\mathcal{G}(\rho)$ has a zero at $\rho = 0$ of multiplicity $q-p$, which is even, and has simple zeros at $\rho_i$, $i = 1, 2, \dots, m_{1}$. 
Consequently, choosing $\beta_1$ sufficiently small in absolute value and with sign opposite to that of $G(0)$, the function $\beta_{1} + \rho^{q-p}\mathcal{G}(\rho)$ has $m_{1}$ simple zeros close to $\rho_1, \dots,  \rho_{m_{1}}$, and, in addition, one positive and one negative simple zero close to the origin $\rho=0$. 
Repeating the previous argument, there exists $\mu_0 \in \mathbb{R}^{2(n+m+2)}$ such that  $M_1(\rho; \mu_0)$ has $m_{1}+2$ simple zeros in $J \setminus \{0\}$.

Case 2: $p>q$. 
We choose $\beta_{2}, \dots, \beta_{m_{1}+2} \in \mathbb{R}$ such that
$$\sum_{k=0}^{m_{1}} \beta_{k+2} \mathcal{T}_{k,\alpha}(y)$$
has exactly $m_{1}$ simple zeros in $J^{'}$, one of them being $y = 0$, and the remaining ones being positive. 
Thus the function $\mathcal{G}(\rho)$ in \eqref{F} has exactly $m_{1}-1$ positive simple zeros, $\rho_2 < \rho_3 < \dots < \rho_{m_{1}}$, when $m_{1}>1$;
when $m_{1}=1$, this list is understood to be empty. 
Moreover, there exists an analytic function $\eta$ such that
$$
\mathcal{G}(\rho) = \sum_{k=0}^{m_{1}} \beta_{k+2} \mathcal{T}_{k,\alpha}(\rho^{p-1}) = \eta(\rho)\rho^{p-1} \quad \text{with } \eta(0) \neq 0.
$$
Without loss of generality, assume that $\eta(0) > 0$. 
Since $p$ is even, there exists a sufficiently small $\hat{\rho}>0$ such that $\rho\mathcal{G}(\rho) > 0$ for all $\rho \in (-\hat{\rho}, \hat{\rho}) \setminus \{0\}$. 
Define
$$
\hat{M}_1(\rho; \beta_{1}, \sigma) := \eta(\rho)\rho^{p-1} + \beta_{1}\rho^{p-q} + \sigma\mathcal{T}_{0,\alpha}(\rho^{p-1}).
$$
We divide the proof into two subcases according to the parity of \(q\).

Subcase 2.1: $q$ odd. 
In this case $\hat{M}_1(\rho; \beta_1, 0) = \rho^{p-q} (\eta(\rho)\rho^{q-1} + \beta_1)$ with $p - q$ odd and $q - 1 > 0$ even. 
Hence, by continuity, there exist sufficiently small $\beta_1^* < 0$ and $\tilde{\rho} \in (0, \hat{\rho})$ such that $\hat{M}_1(\rho; \beta_1^*, 0)$ has $m_{1}-1$ positive simple zeros close to $\rho_2, \dots, \rho_{m_{1}}$, and satisfies
\begin{equation*}
    \quad \quad \hat{M}_1(-\hat{\rho}; \beta_1^*, 0) < 0, \, \hat{M}_1(-\tilde{\rho}; \beta_1^*, 0) > 0, \, \hat{M}_1(\tilde{\rho}; \beta_1^*, 0) < 0, \, \hat{M}_1(\hat{\rho}; \beta_1^*, 0) > 0.
\end{equation*}
Notice that, by \eqref{Definition of T}, $\mathcal{T}_{0,\alpha}(\rho^{p-1}) > 0$ for all $\rho$. 
Therefore, choosing $\sigma^* > 0$ sufficiently small, the function $\hat{M}_1(\rho; \beta_1^*, \sigma^*)$ still has $m_{1}-1$ positive simple zeros close to $\rho_2, \dots, \rho_{m_{1}}$, and satisfies $\hat{M}_1(0; \beta_1^*, \sigma^*) > 0$, together with
\begin{equation*}
    \quad \quad \quad \hat{M}_1(-\hat{\rho}; \beta_1^*, \sigma^*) < 0, \, \hat{M}_1(-\tilde{\rho}; \beta_1^*, \sigma^*) > 0, \, \hat{M}_1(\tilde{\rho}; \beta_1^*, \sigma^*) < 0, \, \hat{M}_1(\hat{\rho}; \beta_1^*, \sigma^*) > 0.
\end{equation*}
Therefore, $\hat{M}_1(\rho; \beta_1^*, \sigma^*)$ has at least one zero in each of the intervals $(-\hat{\rho}, -\tilde{\rho})$, $(0, \tilde{\rho})$, and $(\tilde{\rho}, \hat{\rho})$. 
Hence, by construction, $M_1(\rho; \mu_0) = \rho^q \hat{M}_1(\rho; \beta_1^*, \sigma^*)$ has at least $m_{1}+2$ zeros in $J\setminus\{0\}$. 
Finally, since the upper bound $\mathbb{Z}(M_{1})^{+} + \mathbb{Z}(M_{1})^{-} \le m_{1} + 2$ has already been established, all these zeros must be simple, as desired.

Subcase 2.2: $q$ even. 
In this case $\hat{M}_1(\rho; \beta_1, 0) = \rho^{p-q}(\eta(\rho)\rho^{q-1} + \beta_1)$ with $p - q$ even and $q - 1 > 0$ odd. 
By continuity, there exist sufficiently small $\beta_1^* < 0$ and $\tilde{\rho} \in (0, \hat{\rho})$ such that $\hat{M}_1(\rho; \beta_1^*, 0)$ has $m_{1}-1$ positive simple zeros close to $\rho_2, \dots, \rho_{m_{1}}$ and satisfies
\begin{equation*}
    \hat{M}_1(\hat{\rho}; \beta_1^*, 0) > 0, \quad \hat{M}_1(-\hat{\rho}; \beta_1^*, 0) < 0, \quad \hat{M}_1(\tilde{\rho}; \beta_1^*, 0) < 0.
\end{equation*}
Since $\mathcal{T}_{0,\alpha}(\rho^{p-1}) > 0$ for all $\rho$, choosing $\sigma^* > 0$ sufficiently small ensures that $\hat{M}_1(\rho; \beta_1^*, \sigma^*)$ still has $m_{1}-1$ positive simple zeros close to $\rho_2, \dots, \rho_{m_{1}}$, and additionally satisfies
\begin{equation*}
    \quad \quad \quad \hat{M}_1(\hat{\rho}; \beta_1^*, \sigma^*) > 0, \, \hat{M}_1(-\hat{\rho}; \beta_1^*, \sigma^*) < 0, \, \hat{M}_1(\tilde{\rho}; \beta_1^*, \sigma^*) < 0, \, \hat{M}_1(0; \beta_1^*, \sigma^*) > 0.
\end{equation*}
Therefore, $\hat{M}_1(\rho; \beta_1^*, \sigma^*)$ has at least one zero in each of the intervals $(-\hat{\rho}, 0)$, $(0, \tilde{\rho})$, and $(\tilde{\rho}, \hat{\rho})$. 
Exactly as we did in the previous subcase, there exists $\mu_0 \in \mathbb{R}^{2(n+m+2)}$ such that $M_1(\rho; \mu_0) = \rho^q \hat{M}_1(\rho; \beta_1^*, \sigma^*)$ has exactly $m_{1}+2$ simple zeros in $J \setminus \{0\}$.
\end{itemize}
\end{proof}

\begin{proof}[Proof of Theorem \ref{thm2}]
By applying Proposition \ref{proposition1}, there exists a surjective map $\beta:\mathcal{A}\rightarrow\mathbb{R}^{K+1}$  such that 
    \begin{equation*}
        M_{2}(\rho;\mu)=\rho^{p}\tilde{M}_{2}(\rho;\mu),
    \end{equation*}
where
\begin{equation*}
    \tilde{M}_{2}(\rho;\mu)=\beta_{1}(\mu)+\beta_{2}(\mu)\int_{0}^{T}\tilde{f}^{e}_{0}(t)x_{0}(t,\rho)^{q-p}\,dt+\sum_{i=1}^{K-1}\beta_{i+2}(\mu)\int_{0}^{T}\tilde{F}^{e}_{i}(t)\tilde{f}^{o}_{1}(t)x_{0}(t,\rho)^{q-1}\,dt.
\end{equation*} 
Using identity \eqref{dengs}, we can rewrite $\tilde{M}_{2}(\rho;\mu)$ as
    \begin{equation*}
        \tilde{M}_{2}(\rho;\mu)=\beta_{1}(\mu)+\beta_{2}(\mu)\int_{0}^{T}\tilde{f}^{e}_{0}(t)x_{0}(t,\rho)^{q-p}\,dt+\sum_{i=1}^{K-1}\left(\frac{\beta_{i+2}(\mu)}{p-q}\int_{0}^{T}\tilde{f}^{e}_{i}(t)x_{0}^{q-p}(t,\rho)\,dt\right).
    \end{equation*}
    Since \(p\neq q\), the diagonal change of coordinates
\[
(\beta_1,\beta_2,\beta_3,\ldots,\beta_{K+1})
\longmapsto
\left(
\beta_1,\beta_2,
\frac{\beta_3}{p-q},\ldots,
\frac{\beta_{K+1}}{p-q}
\right)
\]
is invertible. Hence the effective coefficient map remains surjective onto
\(\mathbb R^{K+1}\). The zero-counting and sharpness arguments in the proof
of Theorem~\ref{thm1} therefore apply with \(m_1+1\) replaced by \(K\).
\end{proof}

\begin{proof}[Proof of Theorem \ref{thm3}]
Statement~(i) follows from the first-order analysis by the same argument, so we only prove statement~(ii).
Taking \(p=3\) and \(q=2\), it follows from statement~(ii.2) of Theorem~\ref{thm2} that there exists
\(\mu_{0}\in\mathbb{R}^{2(n+m+2)}\) such that
\[
M_{1}(\rho;\mu_{0})\equiv 0
\]
on \(J\), while \(M_{2}(\rho;\mu_{0})\not\equiv0\) has
\(2(K-1)\) simple zeros in \(J\setminus\{0\}\).
Define the displacement function
\[
D_{\varepsilon}(\rho)
:=
x(T,\rho;\mu_{0},\varepsilon)-\rho .
\]
Since \(M_{1}(\rho;\mu_{0})\equiv0\), we have
\[
D_{\varepsilon}(\rho)
=
\varepsilon^{2}M_{2}(\rho;\mu_{0})
+o(\varepsilon^{2}).
\]
Thus, the normalized displacement function
\[
\widehat D(\rho,\varepsilon)
:=
\frac{D_{\varepsilon}(\rho)}{\varepsilon^{2}},
\qquad \varepsilon\neq0,
\]
extends analytically to \(\varepsilon=0\) by setting
\[
\widehat D(\rho,0):=M_{2}(\rho;\mu_{0}).
\]
Applying the implicit function theorem to each simple zero of
\(M_{2}(\cdot;\mu_{0})\), we conclude that, for every sufficiently
small \(\varepsilon\neq0\), the equation
\[
D_{\varepsilon}(\rho)=0
\]
has at least \(2(K-1)\) nonzero solutions in \(J\setminus\{0\}\).

It remains to prove that the zero solution is also isolated. Since
\(M_{2}(\cdot;\mu_{0})\not\equiv0\), there exists
\(\rho_{*}\in J\) such that
\[
M_{2}(\rho_{*};\mu_{0})\neq0.
\]
Moreover,
\[
\lim_{\varepsilon\to0}
\frac{D_{\varepsilon}(\rho_{*})}{\varepsilon^{2}}
=
M_{2}(\rho_{*};\mu_{0})\neq0.
\]
Consequently, \(D_{\varepsilon}(\rho_{*})\neq0\) for every sufficiently
small \(\varepsilon\neq0\), and hence
\(D_{\varepsilon}\not\equiv0\) for every such \(\varepsilon\).
On the other hand, the zero solution persists under the perturbation, so
\[
D_{\varepsilon}(0)=0.
\]
Since \(D_{\varepsilon}\) is analytic with respect to \(\rho\) and is
not identically zero, its zeros are isolated. In particular,
\(\rho=0\) is an isolated fixed point of the Poincar\'e map and therefore
corresponds to an isolated closed solution.
Hence, for every sufficiently small \(\varepsilon\neq0\), the perturbed
equation has at least
\[
2(K-1)+1=2K-1
\]
limit cycles in \(J\). This proves statement~(ii).
\end{proof}

This completes the proofs of the main results.

\section{Applications to Abel equations with different coefficient families}\label{Applications}
In this section, we apply our Chebyshev framework to Abel equations \eqref{eq1} with trigonometric, polynomial, and hyperbolic coefficient families, respectively.
These coefficient families arise naturally in nonlinear dynamical systems through periodic forcing, polynomial nonlinearities, and exponential-type constitutive relations, respectively.

\subsection{Abel equations with trigonometric coefficients}
Abel equations \eqref{eq1} with trigonometric polynomial coefficients frequently arise in nonlinear systems subject to periodic forcing or parametric excitation. 
For example, periodically forced pendulum models \cite{Strogatz} and Josephson junction dynamics \cite{Barone} involve periodic nonlinearities or external excitations, which are naturally described by trigonometric functions and finite Fourier expansions.
Such systems provide natural motivations for considering Abel equations with trigonometric polynomial coefficients.

We consider equation \eqref{eq1} with $A$ and $B$ being trigonometric polynomials of degrees $n$ and $m$, respectively, that is,
\begin{equation*}
    A(t) =a_{0} + \sum_{i=1}^{n} \left(a_{i}\cos(it) + b_{i}\sin(it)\right), B(t) = c_{0} + \sum_{j=1}^{m} \left(c_{j}\cos(jt) + d_{j}\sin(jt)\right), t\in[-\pi,\pi].
\end{equation*}
Following the notation in the introduction, 
let $\mathcal{H}_{\mathbb{T}}(n,m)$ denote the supremum of the number of limit cycles of this equation. 
Some lower bounds for $\mathcal{H}_{\mathbb{T}}(n,m)$ are already known in \cite{AGY,HTV}. 
In fact, these results are recovered below as consequences of our general theorems.

\begin{theorem}
    Let \(n,m\in\mathbb{N}_{\geq 1}\). 
    The following lower bounds for $\mathcal{H}_{\mathbb{T}}(n,m)$ hold:
    \begin{itemize}
        \item[(i)](\cite{AGY}) $\mathcal{H}_{\mathbb{T}}(n,1)\geq n+2$ and $\mathcal{H}_{\mathbb{T}}(1,m)\geq 2m+1$.
        \item[(ii)](\cite{HTV}) $\mathcal{H}_{\mathbb{T}}(n,m)\geq 2(m+n)-1$.
    \end{itemize}
\end{theorem}
\begin{proof}
Statement~(ii) with $n=1$ coincides with the second half of statement~(i). 
Therefore, it suffices to prove the first half of statement~(i), and the second half follows directly from statement~(ii).
   
Set $\mathcal{F}=\{f_{0},f_{1},f_{2},f_{3},f_{4},\dots\}=\{1, \sin t, \cos t, \sin 2t, \cos 2t, \dots\}$.
We now verify that $\mathcal{F}$ satisfies hypotheses (H$_{1}$), (H$_{2}$) and (H$_{3}$).
\begin{itemize}
    \item[(H$_{1}$)] $f_{1}(t)=\sin t$ is odd and positive on $(0,T)$, so (H$_{1}$) holds.
    \item[(H$_{2}$)] For any $N\in\mathbb{N}$, we have 
                \begin{equation*}
            \mathcal{F}_{N}=
                    \begin{cases}
                        \{1, \sin t, \cos t, \ldots, \sin \frac{N}{2}t, \cos \frac{N}{2}t\}, & \text{if } N \text{ is even},\\
                        \{1, \sin t, \cos t, \ldots, \sin \frac{N+1}{2}t\}, & \text{if } N \text{ is odd}.
                    \end{cases}
                \end{equation*}
            By the definition of $\mathcal{N}_{N}^{\text{non-odd}}$, we have $\#\mathcal{N}_{N}^{\text{non-odd}}=\left[\frac{N}{2}\right]+1$ and
            \begin{equation*}
                \{f_{i}^{e}\}_{i\in\mathcal{N}_{N}^{\text{non-odd}}}=\left\{1,\cos t, \cos 2t,\cdots,\cos \left[\frac{N}{2}\right]t\right\}.
            \end{equation*}
            We now prove that $\{f_{i}^{e}\}_{i\in\mathcal{N}_{N}^{\text{non-odd}}}$ is a CT-system.
            For every $k\in\left\{0,1,\cdots,\left[N/2\right]\right\}$, consider the function
            \begin{equation*}
            f(t) = \sum_{i=0}^k a_i \cos(it),    
            \end{equation*}
            where $a_0,a_1,\ldots,a_k$ are arbitrary real numbers.
            Using the identity $\cos lt = T_l(\cos t)$, where $T_l(x)$ is the Chebyshev polynomial of the first kind of degree $l$, 
            and setting $x = \cos t$, we may write $f(t)$ as a polynomial $P(x)$ in $x$:
            $$P(x) = \sum_{i=0}^k a_i T_i(x).$$
            Since $P(x)$ is a polynomial of degree at most $k$, it has at most $k$ distinct roots.
            The change of variables $x = \cos(t)$ defines a bijection from $(0, \pi)$ onto $(-1, 1)$, 
            which yields a one-to-one correspondence between the roots of $f(t)$ in $(0, \pi)$ and those of $P(x)$ in $(-1, 1)$.
            Consequently, $\{f_i^e\}_{i\in\mathcal{N}_{N}^{\text{non-odd}}}$ is a CT-system, and (H$_2$) holds.
        \item[(H$_{3}$)] In accordance with the setting in Section~\ref{setting}, we introduce the notation
{
\begin{equation*}
    \begin{split}
        &\tilde{f}_{n}^{e}=\cos nt, \quad \tilde{F}_{n}^{e}=\frac{1}{n}\sin nt, \quad \tilde{f}_{n}^{o}=\sin nt,\\
        & \mathcal{B}_{n}=\operatorname{span}\left(\left\{\sin^{2} t,\cdots,\frac{1}{n}\sin nt \sin t\right\}\right).
    \end{split}
\end{equation*}
}

Using the identities
\begin{equation*}
  \sin(n+1)t=\sin t\, U_{n}(\cos t)\ \ \text{and}\ \ U_{k-1}(x)\,U_{l-1}(x)=\sum_{r=0}^{l-1}U_{k-l+2r}(x) \text{ for }k\geq l\geq 1,
\end{equation*}
where $U_i(x)$ denotes the Chebyshev polynomial of the second kind, and exploiting the symmetry in \(k\) and \(l\), we may assume \(k\ge l\ge1\).
Then, we have
            \begin{equation}\label{doln}
                \begin{split}
                    l\tilde{F}_{l}^{e}\tilde{f}_{k}^{o}&= \sin lt \sin kt=\sin ^{2} t\, U_{l-1}(\cos t)U_{k-1}(\cos t)\\
                    &=\sin ^{2} t\sum_{r=0}^{l-1}U_{k-l+2r}(\cos t)=\sin t\, \sum_{r=0}^{l-1}\sin t\, U_{k-l+2r}(\cos t)\\
                    &= \sum_{r=0}^{l-1}\sin\,(k-l+2r+1)t \,\sin t\in \mathcal{B}_{k+l-1}\setminus\mathcal{B}_{k+l-2}.
                \end{split}
            \end{equation}
For coefficient degrees \(n,m\geq1\), the corresponding truncations are \(\mathcal F_{2n}\) and \(\mathcal F_{2m}\). 
Hence $n_1=n$, $m_1=m$, $2m-m_1=m$, and therefore $m_1=m\geq\min\{m-1,n-1\}$.
Thus the numerical part of (H$_3$) also holds.
\end{itemize}
     
We now prove statements (i) and (ii).

    (i) Consider the perturbed Abel equation
\begin{equation}\label{exam1}
\frac{\mathrm{d}x}{\mathrm{d}t}=(\sin t+\varepsilon P_1(t)+\varepsilon^{2} P_2(t))x^2+(\varepsilon Q_1(t)+\varepsilon^{2} Q_2(t))x^3,
\end{equation}
where $P_i(t)$ and $Q_i(t)$, $i=1,2$, are trigonometric polynomials of degrees $1$ and $n$, respectively.
This shows that $Q_{i}(t)\in\operatorname{span}(\mathcal{F}_{2n})=\operatorname{span}\left(\{1,\sin t, \cos t, \dots, \sin nt, \cos nt\}\right)$, and $\#\mathcal{N}_{2n}^{\text{non-odd}}=n+1$.
Taking $p=2$ and $q=3$, it follows from statement (iii) of Theorem \ref{thm1} and the same argument in the proof of statement (ii) of Theorem \ref{thm3} that, for every sufficiently small $\varepsilon\ne0$, equation \eqref{exam1} possesses at least $n+1$ nonzero limit cycles.
Therefore $\mathcal{H}_{\mathbb{T}}(n,1)\geq n+2$.

    (ii) Consider the perturbed Abel equation
\begin{equation}\label{exam2}
    \frac{dx}{dt}=(\sin t+\varepsilon P_{1}(t)+\varepsilon^{2} P_{2}(t))x^{3}+(\varepsilon Q_{1}(t)+\varepsilon^{2} Q_{2}(t))x^{2},
\end{equation}
where $P_i(t)$ and $Q_i(t)$, $i=1,2$, are trigonometric polynomials of degrees $n$ and $m$, respectively.
By statement (ii) of Theorem \ref{thm3}, $\mathcal{H}_{\mathbb{T}}(n,m)\geq 2(m+n)-1$.
\end{proof}

\subsection{Abel equations with polynomial coefficients}
Polynomial-coefficient Abel equations are closely related to polynomial nonlinear oscillators. 
In particular, certain Li\'{e}nard-type equations can be transformed into first-order Abel equations, and this reduction includes generalized Van der Pol-type models (see \cite{Harko2014}). 
Polynomial nonlinearities also occur in classical oscillators such as the Duffing system \cite{GuckenheimerHolmes}.

As a second application, we consider equation~\eqref{eq1} with $A$ and $B$ being polynomials of degrees $n$ and $m$, respectively, that is,
\begin{equation*}
A(t)=\sum_{i=0}^{n}a_{i}t^{i}\quad\text{and}\quad B(t)=\sum_{j=0}^{m}b_{j}t^{j},\quad\ \ t\in[-T,T].
\end{equation*}
Similar to the trigonometric case, we denote by $\mathcal{H}_{\mathbb{P}}(n,m)$ the supremum of the number of limit cycles of equation~\eqref{eq1} in this setting.
To the best of our knowledge, the only known result concerning $\mathcal{H}_{\mathbb{P}}(n,m)$ is due to Lins-Neto \cite{Lins}, who proved that $$\mathcal{H}_{\mathbb{P}}(m,m)\geq m.$$
We obtain the following improved result.
\begin{theorem}
    Let \(n,m\geq2\). Then $\mathcal{H}_{\mathbb{P}}(n,m)\geq 2\left(m-\left[\frac{m}{2}\right]+\left[\frac{n}{2}\right]\right)-1$.
\end{theorem}

\begin{proof}
Without loss of generality, we may rewrite $A$ and $B$ in equation~\eqref{eq1} as
\begin{equation*}
A(t)=\alpha_{0}+\sum_{i=1}^{n}a_{i}\left(t^{i}-C_{i}\right)\quad\text{and}\quad B(t)=\beta_{0}+\sum_{j=1}^{m}b_{j}\left(t^{j}-C_{j}\right),
\end{equation*}
where $t\in[-T,T]$, $\alpha_{0}=a_{0}+\sum_{i=1}^{n}a_{i}C_{i}$, $\beta_{0}=b_{0}+\sum_{j=1}^{m}b_{j}C_{j}$, $C_{2k}=\frac{T^{2k}}{2k+1}$ and $C_{2k-1}=0$ for $k\in\mathbb{N}_{\geq 1}$.
Set $\mathcal{F}=\{f_0,f_1,f_2,f_3,f_4,\dots\}=\{1,t,t^2-C_2,t^3,t^4-C_4,\dots\}$.
Clearly, we have $\int_{-T}^{T}f_{i}(t)dt=0$ for $i\in\mathbb{N}_{\geq 1}$ and $A\in\operatorname{span}(\mathcal F_n)$, $B\in\operatorname{span}(\mathcal F_m)$.

Next, we verify that $\mathcal{F}$ satisfies hypotheses (H$_1$), (H$_2$) and (H$_3$).

\begin{itemize}
    \item[(H$_1$)] $f_1(t)=t$ is odd and positive on $(0,T)$, so (H$_1$) holds.
    
    \item[(H$_2$)] For any $N\in\mathbb{N}$, we have
    \begin{equation*}
    \mathcal{F}_N=
    \begin{cases}
    \{1,t,t^2-C_2,\dots,t^N-C_N\}, & \text{if $N$ is even},\\[4pt]
    \{1,t,t^2-C_2,\dots,t^N\}, & \text{if $N$ is odd}.
    \end{cases}
    \end{equation*}
    By the definition of $\mathcal{N}_N^{\text{non-odd}}$, we have $\#\mathcal{N}_N^{\text{non-odd}}=\left[\frac{N}{2}\right]+1$ and
    \begin{equation*}
    \{f_i^e\}_{i\in\mathcal{N}_N^{\text{non-odd}}}=\bigl\{1,t^2-C_2,t^4-C_4,\dots,t^{2\left[\frac{N}{2}\right]}-C_{2\left[\frac{N}{2}\right]}\bigr\}.
    \end{equation*}
    We now prove that $\{f_i^e\}_{i\in\mathcal{N}_N^{\text{non-odd}}}$ is a CT-system.
    For every $k\in\{0,1,\dots,\left[\frac{N}{2}\right]\}$, observe that
    \begin{equation*}
    \operatorname{span}\bigl(\{1,t^2-C_2,\dots,t^{2k}-C_{2k}\}\bigr)=\operatorname{span}\bigl(\{1,t^2,\dots,t^{2k}\}\bigr).
    \end{equation*}
    Since $\{1,t^2,\dots,t^{2k}\}$ is a CT-system on $(0,T)$, it follows that $\{f_i^e\}_{i\in\mathcal{N}_N^{\text{non-odd}}}$ is a CT-system on $(0,T)$.
    Hence (H$_2$) holds.
    
    \item[(H$_3$)] In accordance with the setting in Section~\ref{setting}, we introduce the notation
    {
\begin{equation*}
    \begin{split}
        &\tilde{f}_{n}^{e}=t^{2n}-C_{2n}, \quad \tilde{F}_{n}^{e}=\tfrac{1}{2n+1}t^{2n+1}-C_{2n}t, \quad \tilde{f}_{n}^{o}=t^{2n-1},\\
        &\mathcal{B}_n=\operatorname{span}\left(\left\{\tfrac{1}{3}t^4-C_2t^2,\dots,\tfrac{1}{2n+1}t^{2n+2}-C_{2n}t^2\right\}\right).
    \end{split}
\end{equation*}
}
    Note that for all $l\in\mathbb{Z}_{\geq 1}$ and $k=1$,
\begin{equation*}
    \begin{split}
    \tilde{F}_l^e\tilde{f}_1^o=&\frac{1}{2l+1}t^{2l+2}-C_{2l}t^{2}\in\mathcal{B}_{l}\setminus\mathcal{B}_{l-1}.
    \end{split}
\end{equation*}
Moreover, for all \(l\in\mathbb{Z}_{\geq 1}\) and $k\in\mathbb{Z}_{\geq 2}$,
    \begin{equation*}
    \begin{split}
    \tilde{F}_l^e\tilde{f}_k^o=&\frac{1}{2l+1}t^{2k+2l}-C_{2l}t^{2k}\\
    =&\frac{2k+2l-1}{2l+1}\Bigl(\frac{1}{2k+2l-1}t^{2k+2l}-C_{2k+2l-2}t^2\Bigr)\\
    &-(2k-1)C_{2l}\Bigl(\frac{1}{2k-1}t^{2k}-C_{2k-2}t^2\Bigr)\in\mathcal{B}_{k+l-1}\setminus\mathcal{B}_{k+l-2}.
    \end{split}
    \end{equation*}
For the truncations \(\mathcal F_n\) and \(\mathcal F_m\),
$n_1=\left[\frac{n}{2}\right]$, $m_1=\left[\frac{m}{2}\right]$, $m-m_1=\left[\frac{m}{2}\right]$.
If \(n,m\geq2\), then $m_1
\geq
\min\left\{
\left\lceil\frac{m}{2}\right\rceil-1,
\left[\frac{n}{2}\right]-1
\right\}$.
Thus the numerical part of hypothesis~(H$_3$) is satisfied.
    Hence (H$_3$) holds.
\end{itemize}

Consider the perturbed Abel equation
\begin{equation}\label{exam3}
\frac{\mathrm{d}x}{\mathrm{d}t}=(t+\varepsilon P_1(t)+\varepsilon^{2} P_2(t))x^3+(\varepsilon Q_1(t)+\varepsilon^{2} Q_2(t))x^2,
\end{equation}
where $P_i(t)$ and $Q_i(t)$ for $i=1,2$, are polynomials of degrees $n$ and $m$, respectively, that is,
\begin{equation*}
P_i(t)=\sum_{j=0}^{n}a_{ij}t^{j},\quad Q_i(t)=\sum_{j=0}^{m}b_{ij}t^{j},\quad i=1,2.
\end{equation*}
By statement (ii) of Theorem~\ref{thm3}, $\mathcal{H}_{\mathbb{P}}(n,m)\geq 2\bigl(m-\left[\frac{m}{2}\right]+\left[\frac{n}{2}\right]\bigr)-1$.
\end{proof}

\subsection{Abel equations with hyperbolic coefficients}
Beyond polynomial and trigonometric coefficient functions, Abel-type equations with exponential factors also arise in reductions of certain nonlinear partial differential equations (see \cite{Mhadhbi2024}). 
Since hyperbolic functions are linear combinations of exponentials, this motivates examining hyperbolic coefficient families as a representative non-polynomial test case for our Chebyshev framework.
Recall that
\begin{equation*}
    \sinh t=\frac{e^{t}-e^{-t}}{2}\quad\text{and} \quad\cosh t=\frac{e^{t}+e^{-t}}{2}.
\end{equation*}

We consider equation~\eqref{eq1} with hyperbolic polynomial coefficients and set
\begin{equation*}
    \begin{split}
            A(t)& =a_{0} + \sum_{i=1}^{n} \left(a_{i}\cosh (it) + b_{i}\sinh (it)\right),\\ 
            B(t)& = c_{0} + \sum_{j=1}^{m} \left(c_{j}\cosh (jt) + d_{j}\sinh (jt)\right).
    \end{split}
\end{equation*}
Similarly, we denote by $\mathcal{H}_{\mathbb{H}}(n,m)$ the supremum of the number of limit cycles of this equation.
As we did in the previous subsection, without loss of generality, we may write $A$ and $B$ in the following form:
\begin{equation*}
    \begin{split}
            A(t)&=\alpha + \sum_{i=1}^{n} \left(a_{i}\left(\cosh (it)-C_{i}\right) + b_{i}\sinh (it)\right),\\
            B(t)&= \beta + \sum_{j=1}^{m} \left(c_{j}\left(\cosh (jt)-C_{j}\right) + d_{j}\sinh (jt)\right),
    \end{split}
\end{equation*}
where $t\in[-T,T]$, $C_{k}=\frac{\sinh (kT)}{kT}$, $\alpha=a_{0}+\sum_{i=1}^{n}a_{i}C_{i}$ and $\beta=c_{0}+\sum_{j=1}^{m}c_{j}C_{j}$.

Set $\mathcal{F}=\{f_{0},f_{1},f_{2},f_{3},f_{4},\cdots\}=\{1, \sinh t, \cosh t-C_{1}, \sinh (2t), \cosh (2t)-C_{2}, \ldots\}$.
By verifying that $\mathcal{F}$ satisfies hypotheses (H$_{1}$) and (H$_{2}$), we have the following result:
\begin{theorem}
    For $n,m\in\mathbb N_{\geq1}$, one has
\[
\mathcal H_H(n,m)\geq 2m+1.
\]
\end{theorem}
\begin{proof}
    \begin{itemize}
        \item[(H$_{1}$)] Clearly, $f_{1}(t)=\sinh t$ is odd and positive on $(0,T)$. Hence (H$_{1}$) holds.
        \item[(H$_{2}$)] For any $N\in\mathbb{N}$, we have 
                \begin{equation*}
            \mathcal{F}_{N}=
                    \begin{cases}
                        \{1, \sinh t, \cosh t-C_{1}, \ldots, \sinh \left(\frac{N}{2}t\right), \cosh \left(\frac{N}{2}t\right)-C_{\frac{N}{2}}\}, & \text{if } N \text{ is even},\\
                        \{1, \sinh t, \cosh t-C_{1}, \ldots, \sinh \left(\frac{N+1}{2}t\right)\}, & \text{if } N \text{ is odd}.
                    \end{cases}
                \end{equation*}
            From the definition of $\mathcal{N}_{N}^{\text{non-odd}}$ we obtain that $\#\mathcal{N}_{N}^{\text{non-odd}}=\left[\frac{N}{2}\right]+1$ and
            \begin{equation*}
                \{f_{i}^{e}\}_{i\in\mathcal{N}_{N}^{\text{non-odd}}}=\left\{1,\cosh t-C_{1}, \cosh 2t-C_{2},\cdots,\cosh \left(\left[\frac{N}{2}\right]t\right)-C_{\left[\frac{N}{2}\right]}\right\}.
            \end{equation*}
            We now prove that $\{f_{i}^{e}\}_{i\in\mathcal{N}_{N}^{\text{non-odd}}}$ is a CT-system.
            For every $k\in\left\{0,1,\cdots,\left[\frac{N}{2}\right]\right\}$, consider the function
            \begin{equation*}
            f(t) = \sum_{i=0}^k a_i \cosh (it),    
            \end{equation*}
            where $a_0,a_1,\ldots,a_k$ are arbitrary real numbers.
            We use the identity $\cosh lt = T_l(\cosh t)$, where $T_l(x)$ is the Chebyshev polynomial of the first kind of degree $l$. 
            By making the change of variable $x = \cosh t$, we can express $f(t)$ as a polynomial $P(x)$ in $x$:
            $$P(x) = \sum_{i=0}^k a_i T_i(x).$$
            Clearly, $P(x)$ is a polynomial of degree at most $k$ and has at most $k$ distinct roots.
            The mapping $x = \cosh(t)$ is a bijection between the interval $t \in (0, T)$ and the interval $x \in (1, \cosh T)$. 
            This implies a one-to-one correspondence between the roots of $f(t)$ in $(0, T)$ and the roots of $P(x)$ in $(1, \cosh T)$.
            Thus, $\{f_{i}^{e}\}_{i\in\mathcal{N}_{N}^{\text{non-odd}}}$ is a CT-system and (H$_{2}$) holds.
    \end{itemize}

    Consider the perturbed Abel equation
    \begin{equation}\label{exam4}
        \frac{dx}{dt}=\left(\sinh t+\varepsilon P_{1}(t)+\varepsilon^{2} P_{2}(t)\right)x^{3}+\left(\varepsilon Q_{1}(t)+\varepsilon^{2} Q_{2}(t)\right)x^{2},
    \end{equation}
    where 
    \begin{equation*}
        P_{i}(t)=a_{i0} + \sum_{j=1}^{n} \left(a_{ij}\cosh (jt) + b_{ij}\sinh (jt)\right),\quad Q_{i}(t)=c_{i0} + \sum_{j=1}^{m} \left(c_{ij}\cosh (jt) + d_{ij}\sinh (jt)\right).
    \end{equation*}
Theorem \ref{thm3} (i) yields
$\mathcal{H}_{\mathbb{H}}(n,m)\geq 2m+1$.
\end{proof}
Note that in the proof of the above result, we only verify hypotheses (H$_{1}$) and (H$_{2}$). 
In fact, the hyperbolic family fails (H$_3$), as shown next.

\begin{proposition}
    The set $\mathcal{F}=\{f_{0},f_{1},f_{2},f_{3},f_{4},\dots\}=\{1, \sinh t, \cosh t-C_{1}, \sinh (2t), \cosh (2t)-C_{2}, \ldots\}$ does not satisfy hypothesis (H$_{3}$).
\end{proposition}
\begin{proof}
We now prove that  
\(\tilde{F}_1^e(t) \tilde{f}_2^o(t)\) does not belong to $\operatorname{span}\left(\left\{\tilde{F}_1^e\tilde{f}_1^o,\tilde{F}_2^e\tilde{f}_1^o\right\}\right)$.
Note that
\begin{equation*}
    \begin{split}
        &\tilde{F}_1^e(t) \tilde{f}_2^o(t) = \left(\sinh t - C_1 t\right) \sinh(2t)=\frac{e^{-3t}}{4}-\frac{e^{-t}}{4}-\frac{e^{t}}{4}+\frac{e^{3t}}{4}+\frac{C_{1}}{2}t e^{-2t}-\frac{C_{1}}{2}t e^{2t},\\
        &\tilde{F}_1^e(t) \tilde{f}_1^o(t) = \left(\sinh t - C_1 t\right)\sinh t =-\frac12+\frac{e^{-2t}}{4}+\frac{e^{2t}}{4}+\frac{C_{1}}{2}t e^{-t}-\frac{C_{1}}{2}t e^{t},\\
        &\tilde{F}_2^e(t) \tilde{f}_1^o(t) = \left(\frac{\sinh(2t)}{2} - C_2 t\right)\sinh t = \frac{e^{-3t}}{8}-\frac{e^{-t}}{8}-\frac{e^{t}}{8}+\frac{e^{3t}}{8}+\frac{C_{2}}{2}t e^{-t}-\frac{C_{2}}{2}t e^{t}.
    \end{split}
\end{equation*}
Since $C_1\neq0$ and
$
te^{-2t},\,te^{2t}
\notin
\operatorname{span}
\left\{
e^{-3t},e^{-2t},e^{-t},
e^{t},e^{2t},e^{3t},
te^{-t},te^{t}
\right\},$
it follows that
$$
\tilde{F}_1^e(t)\tilde{f}_2^o(t)
\notin
\operatorname{span}
\left(
\left\{
\tilde{F}_1^e\tilde{f}_1^o,
\tilde{F}_2^e\tilde{f}_1^o
\right\}
\right).$$
Therefore, the hyperbolic family $\mathcal F$ fails to satisfy hypothesis (H$_3$). This completes the proof.
\end{proof}

\section{Conclusions and future work}

In this paper, we have studied the maximum number of limit cycles of generalized Abel equations with coefficients belonging to finite-dimensional spaces generated by Chebyshev families. 
We developed a unified framework based on first- and second-order Melnikov analyses together with the zero-counting properties of Chebyshev systems.
More precisely, by introducing suitable hypotheses on the symmetry and algebraic structure of the coefficient spaces, we obtained explicit estimates for the number of zeros of the first- and second-order Melnikov functions. These estimates lead to lower bounds for the corresponding Hilbert numbers of generalized Abel equations. Our results extend previous studies restricted to particular coefficient families, such as polynomial and trigonometric polynomial coefficients, to a broader class of functions satisfying suitable Chebyshev properties.

The examples presented in Section \ref{Applications} demonstrate the applicability of our framework to several representative coefficient families. In particular, we recover known results for trigonometric polynomial coefficients, improve previous lower bounds in the polynomial setting, and illustrate both the applicability and the limitations of our method for hyperbolic coefficient functions. These examples show that the Chebyshev approach provides a flexible framework for studying Abel equations beyond classical polynomial and trigonometric cases.

As the final part of this paper, we briefly discuss several possible directions for future research.

\begin{itemize}
    \item[(1)] \textbf{Higher-order Melnikov analysis.}
    The present work develops the Chebyshev framework up to the second-order Melnikov function. 
    Extending this approach to higher-order Melnikov functions is a natural but challenging direction. 
    Such an extension requires a deeper understanding of the algebraic structure of the function spaces generated by higher-order products of coefficient functions and their primitives.
    A concrete first step would be to study the third-order Melnikov function $M_{3}$ for the perturbed Abel equation.
\begin{equation*}
\frac{dx}{dt}
=
\left(\sin t+\varepsilon P_1(t)+\varepsilon^2P_2(t)+\varepsilon^3P_3(t)\right)x^3
+
\left(\varepsilon Q_1(t)+\varepsilon^2Q_2(t)+\varepsilon^3Q_3(t)\right)x^2,    
\end{equation*}
where $P_i(t)$ and $Q_i(t)$ are trigonometric polynomials of degrees $n$ and $m$, respectively, for $i=1,2,3$.
    \item[(2)] \textbf{Weaker hypotheses on coefficient spaces.}
    The hypotheses imposed in this paper, especially (H$_3$), provide sufficient conditions to guarantee the Chebyshev structure required for the second-order analysis. 
    However, these assumptions may not be necessary. 
    It would be interesting to develop weaker conditions or alternative approaches that apply to coefficient families not satisfying (H$_3$), such as the hyperbolic family discussed in Section \ref{Applications}.

    \item[(3)] \textbf{Sharper estimates and exact Hilbert numbers.}
    The results obtained in this paper provide lower bounds for the maximum number of limit cycles. 
    A challenging problem is to determine whether these bounds are optimal and to obtain exact values of the corresponding Hilbert numbers of generalized Abel equations.
\end{itemize}
\section*{Acknowledgements}
We thank the anonymous referee for their valuable comments and suggestions,
which helped us improve both the mathematical content and the presentation
of this paper.

The first author is supported by NNSF of China (No. 12271212).
The second and third authors are supported by the NNSF of China (No. 12371183 and No. 124B2006).

\section*{Author Contributions}
All authors reviewed the manuscript.

\section*{Availability of data and materials}
No datasets were generated or analysed during the current study.

\section*{Conflict of Interests}
The authors declare that they have no conflict of interests regarding the publication of this paper.


\begin{thebibliography}{99}

\bibitem{ABS} {\sc A. \'{A}lvarez, J. L. Bravo and F. Sánchez.} 
{\it\ Planar systems and Abel equations.} Commun. Pure Appl. Anal. 
{\bf 21} (2022), 3463--3478.

\bibitem{AGY} {\sc M. J. \'{A}lvarez, A. Gasull and J. Yu.} 
{\it\ Lower bounds for the number of limit cycles of trigonometric Abel equations.} 
J. Math. Anal. Appl. {\bf 342} (2008), 682--693.

\bibitem{Alvarez2007} {\sc M. J. \'{A}lvarez, A. Gasull and H. Giacomini.} 
{\it\ A new uniqueness criterion for the number of periodic orbits of Abel equations.} 
J. Differ. Equ. {\bf 234} (2007), 161--176.

\bibitem{Arnold1} {\sc V. I. Arnol'd.} 
{\it\ Loss of stability of self-induced oscillations near resonance, and versal deformations of equivariant vector fields.} 
Funct. Anal. Appl. {\bf 11} (1977), 85--92.

\bibitem{Arnold2} {\sc V. I. Arnol'd.} 
{\it\ Ten problems.} In \textit{Theory of singularities and its applications}, 
volume 1 of \textit{Adv. Soviet Math.}, pages 1--8. 
Amer. Math. Soc., Providence, RI, 1990.

\bibitem{Barone} {\sc A. Barone and G. Paternò.} 
{\it\ Physics and Applications of the Josephson Effect},
John Wiley \& Sons, New York, 1982.

\bibitem{bravo1} {\sc J. L. Bravo, M. Fern\'{a}ndez and A. Gasull.} 
{\it\ Stability of singular limit cycles for Abel equations.} 
Discrete Contin. Dyn. Syst. {\bf 35} (2015), 1873--1890.

\bibitem{bravo2} {\sc J. L. Bravo, M. Fern\'{a}ndez and I. Ojeda.} 
{\it\ Stability of singular limit cycles for Abel equations revisited.} 
J. Differential Equations {\bf 379} (2024), 1--25.

\bibitem{Cherkas} {\sc L. A. Cherkas.} 
{\it\ Conditions for a center for a certain Liénard equation}. 
Differ. Uravn., {\bf 12} (1976), 292--298.

\bibitem{appliction1} {\sc E. Fossas, J. M. Olm and H. Sira-Ram\'{\i}rez.} 
{\it\ Iterative approximation of limit cycles for a class of Abel equations.} 
Phys. D {\bf 237} (2008), 3159--3164.

\bibitem{GG} {\sc A. Gasull and A. Guillamon.} 
{\it\ Limit cycles for generalized Abel equations.} 
Int. J. Bifurc. Chaos {\bf 16} (2006), 3737--3745.

\bibitem{GGM} {\sc A. Gasull, A. Geyer and F. Mañosas.} 
{\it\ A Chebyshev criterion with applications.} 
J. Differential Equations {\bf 269} (2020), 6641--6655.

\bibitem{GL} {\sc A. Gasull and J. Llibre.} 
{\it\ Limit cycles for a class of Abel equations.} 
SIAM J. Math. Anal. {\bf 21} (1990), 1235--1244.

\bibitem{openproblem} {\sc A. Gasull.} 
{\it\ Some open problems in low dimensional dynamical systems.} 
SeMA J. {\bf 78} (2021), 233--269.

\bibitem{GuckenheimerHolmes} {\sc J. Guckenheimer and P. Holmes.} 
{\it\ Nonlinear Oscillations, Dynamical Systems, and Bifurcations of Vector Fields.} 
Applied Mathematical Sciences, vol. 42, Springer-Verlag, New York, 1983.

\bibitem{Harko2014} {\sc T. Harko, F. S. N. Lobo and M. K. Mak.} 
{\it\ A class of exact solutions of the Li\'{e}nard-type ordinary nonlinear differential equation.} 
J. Eng. Math. {\bf 89} (2014), 193--205.

\bibitem{appliction2} {\sc T. Harko and M. K. Mak.} 
{\it\ Relativistic dissipative cosmological models and Abel differential equation.} 
Comput. Math. Appl. {\bf 46} (2003), 849--853.

\bibitem{appliction3} {\sc T. Harko and M. K. Mak.} 
{\it\ Travelling wave solutions of the reaction-diffusion mathematical model of glioblastoma growth: An Abel equation based approach.} 
Math. Biosci. Eng. {\bf 12} (2015), 41--69.

\bibitem{HL} {\sc J. Huang and H. Liang.} 
{\it\ Estimate for the number of limit cycles of Abel equation via a geometric criterion on three curves.} 
Nonlinear Differ. Equ. Appl. {\bf 24} (2017), 47.

\bibitem{HTV} {\sc J. Huang, J. Torregrosa and J. Villadelprat.} 
{\it\ On the number of limit cycles in generalized Abel equations.} 
SIAM J. Appl. Dyn. Syst. {\bf 19} (2020), 2343--2370.

\bibitem{HTZ} {\sc J. Huang, R. Tian and Y. Zhao.} 
{\it\ A Chebyshev criterion for at most two non-zero limit cycles in Abel equations.} 
Nonlinearity {\bf 39} (2026), 015030.

\bibitem{Ily} {\sc Y. Ilyashenko.} 
{\it\ Centennial history of Hilbert's 16th problem.} 
Bull. Amer. Math. Soc. (N.S.) {\bf 39} (2002), 301--354.

\bibitem{KS}
{\sc S. Karlin and W. J. Studden.}
{\it Tchebycheff Systems: With Applications in Analysis and Statistics.}
Pure and Applied Mathematics, Vol. XV, Interscience Publishers,
New York--London--Sydney, 1966.

\bibitem{Kru} {\sc M. Krusemeyer.} 
{\it\ Why does the Wronskian work?} 
Am. Math. Mon. {\bf 95} (1988), 46--49.

\bibitem{Li} {\sc J. Li.} 
{\it\ Hilbert's 16th problem and bifurcations of planar polynomial vector fields.} 
Internat. J. Bifur. Chaos Appl. Sci. Engrg. {\bf 13} (2003), 47--106.

\bibitem{Lins} {\sc A. Lins-Neto.} 
{\it\ On the number of solutions of the equation $dx/dt=\sum_{j=0}^{n}a_j(t)x^j$, $0\leq t\leq1$ for which $x(0)=x(1)$.} 
Invent. Math. {\bf 59} (1980), 67--76.

\bibitem{Lloyd2} {\sc N. G. Lloyd.} 
{\it\ On a class of differential equations of Riccati type.} 
J. London Math. Soc. {\bf 10} (1975), 1--10.

\bibitem{Lloyd} {\sc N. G. Lloyd.} 
{\it\ A note on the number of limit cycles in certain two-dimensional systems.} 
J. London Math. Soc. {\bf 20} (1979), 277--286.

\bibitem{Maz} {\sc M.-L. Mazure.} 
{\it\ Chebyshev spaces and Bernstein bases.} 
Constr. Approx. {\bf 22} (2005), 347--363.

\bibitem{Mhadhbi2024} {\sc N. Mhadhbi, S. Gana and M. F. Alsaeedi.} 
{\it\ Exact solutions for nonlinear partial differential equations via a fusion of classical methods and innovative approaches.} 
Sci. Rep. {\bf 14} (2024), 6443.

\bibitem{Rou} {\sc R. Roussarie.} 
{\it\ Bifurcations of Planar Vector Fields and Hilbert's Sixteenth Problem.} 
Progr. Math. 164, Birkh\"auser Verlag, Basel, 1998.

\bibitem{Smale0} {\sc S. Smale.} 
{\it\ Dynamics retrospective: great problems, attempts that failed.} 
Physica D {\bf 51} (1991), 267--273.

\bibitem{Smale1} {\sc S. Smale.} 
{\it\ Mathematical problems for the next century.} 
Math. Intell. {\bf 20} (1998), 7--15.

\bibitem{Strogatz} {\sc S. H. Strogatz.} 
{\it\ Nonlinear Dynamics and Chaos: With Applications to Physics, Biology, Chemistry, and Engineering, 2nd ed.} 
Boca Raton, FL, USA: CRC Press, 2015.

\bibitem{appliction4} {\sc A. V. Yurov, A. V. Yaparova and V. A. Yurov.} 
{\it\ Application of the Abel equation of the 1st kind to inflation analysis of non-exactly solvable cosmological models.} 
Gravit. Cosmol. {\bf 20} (2014), 106--115.

\bibitem{YHL} {\sc X. Yu, J. Huang and C. Liu.} 
{\it\ Maximum number of limit cycles for Abel equation having coefficients with linear trigonometric functions.} 
J. Differ. Equ. {\bf 410} (2024), 301--318.
    \end{thebibliography}
\end{document}